\documentclass[12pt]{amsart}

\usepackage{amsmath}
\usepackage{amssymb}
\usepackage{amsthm}
\usepackage{amscd}
\usepackage{mathrsfs}
\usepackage[final]{showlabels}

\theoremstyle{plain}
\newtheorem{thm}{Theorem}[subsection]
\newtheorem{prop}[thm]{Proposition}
\newtheorem{lem}[thm]{Lemma}
\newtheorem{cor}[thm]{Corollary}

\newtheorem{thmintro}{Theorem}

\newtheorem*{conjintro}{Conjecture}

\theoremstyle{definition}
\newtheorem{definition}[thm]{Definition}

\newtheorem{example}[thm]{Example}
\newtheorem*{acknowledgement}{Acknowledgement}

\theoremstyle{remark}
\newtheorem{rem}[thm]{Remark}

\numberwithin{equation}{section}

\DeclareMathOperator{\Gal}{Gal}
\DeclareMathOperator{\rankZ}{rank_{\mathbb{Z}}}
\DeclareMathOperator{\rankZp}{rank_{\mathbb{Z}_\mathit{p}}}
\DeclareMathOperator{\rankp}{rank_\mathit{p}}
\DeclareMathOperator{\ranktwo}{rank_{2}}
\DeclareMathOperator{\Gnorm}{\mathit{N} \mskip -2mu}
\DeclareMathOperator{\Lnorm}{\mathcal{N} \mskip -2mu}

\begin{document}

\title{Abelian $p$-extensions with restricted $p$-ramification 
and the cyclotomic $\mathbb{Z}_2$-extension of $\mathbb{Q} (\sqrt{q})$}
\footnote[0]{2020 Mathematics Subject Classification. 11R23, 11R11}
\author{Tsuyoshi Itoh and Naoki Kumakawa}
%\date{\today}

\begin{abstract}
We develop Hachimori's study on the unramified Iwasawa modules 
using extensions with restricted $p$-ramification.
Let $K$ be an algebraic number field, and $p$ a prime number which 
splits into two distinct primes $\mathfrak{p}$, $\mathfrak{p}'$ in $K$.
Assume that $K$ and $p$ satisfies several (somewhat strict) conditions.  

Let $K_\infty /K$ be the cyclotomic $\mathbb{Z}_p$-extension.
In the present paper, we give a method to study the structure of 
the unramified Iwasawa module $X (K_\infty)$ by using 
abelian $p$-extensions unramified outside $\mathfrak{p}$.
We give a sufficient condition for $|X (K_\infty)|$ to be finite  
in terms of such extensions.
We also give a similar criterion for $X (K_\infty)$ to be finitely generated 
over $\mathbb{Z}_p$.

In the latter part of the present paper, we consider the case where 
$k = \mathbb{Q} (\sqrt{q})$ with an odd prime number $q$ and 
apply our results to the cyclotomic $\mathbb{Z}_2$-extension $k_\infty /k$. 
We give a necessary and sufficient condition for $X (k_\infty)$ to be cyclic over $\mathbb{Z}_2$, 
which is different from the former results given by either Mouhib-Movahhedi 
or Mizusawa-Mouhib.
We also give several sufficient conditions for the validity of Greenberg's conjecture.
\end{abstract}

\maketitle

\section{Introduction}\label{section_intro}
We first recall the Iwasawa invariants of the cyclotomic $\mathbb{Z}_p$-extension of 
algebraic number fields, and several conjectures concerning these invariants.

Let $p$ be a fixed prime number.
For an algebraic number field $\mathbb{K}$, 
we denote by $\mathbb{K}_\infty / \mathbb{K}$ the cyclotomic $\mathbb{Z}_p$-extension. 
We also denote by $\mathbb{K}_n$ the $n$th layer of $\mathbb{K}_\infty / \mathbb{K}$ 
for a non-negative integer $n$ 
($\mathbb{K}_n /\mathbb{K}$ is the unique cyclic subextension of degree $p^n$).
For a set $S$, we denote by $|S|$ the cardinality of $S$.

As is well known, 
the Iwasawa invariants $\lambda$, $\mu$, $\nu$ of $\mathbb{K}_\infty / \mathbb{K}$ are 
defined, and these invariants describe the behavior of the order of the 
Sylow $p$-subgroup $A (\mathbb{K}_n)$ of the ideal class group of $\mathbb{K}_n$.
That is, the formula 
\[ |A (\mathbb{K}_n)| = p^{\lambda n + \mu p^n + \nu} \]
holds for all sufficiently large $n$ (see, e.g., \cite{Iwa73}).
This is called Iwasawa's class number formula.
Recall that $\lambda$ and $\mu$ deeply relate to the structure of 
the unramified Iwasawa module $X (\mathbb{K}_\infty)$, 
which is the Galois group of the maximal unramified abelian pro-$p$ extension 
of $\mathbb{K}_\infty$.

Concerning the Iwasawa invariants, the following conjectures are widely known. 

\begin{conjintro}[{Greenberg's conjecture \cite{Gre76}}]
If $\mathbb{K}$ is totally real, then $\lambda = \mu =0$.
\end{conjintro}

\begin{conjintro}[{$\mu=0$ conjecture \cite{Iwa73mu}}] 
$\mu =0$ for every $\mathbb{K}$.
\end{conjintro}

It is also well known that the following are equivalent:
\begin{itemize}
\item[(i)] $\lambda  = \mu =0$,  
\item[(ii)] $|X (\mathbb{K}_\infty)|$ is finite,
\item[(iii)] $|A (\mathbb{K}_n)|$ is bounded.
\end{itemize}
For the triviality of $\mu$, it is known that the following are equivalent:
\begin{itemize}
\item[(i)] $\mu =0$, 
\item[(ii)] $X (\mathbb{K}_\infty)$ is finitely generated over $\mathbb{Z}_p$,
\item[(iii)] the $p$-rank of $A (\mathbb{K}_n)$ is bounded.
\end{itemize}
(For the above results, see, e.g., \cite{Gre76}, \cite{Iwa73}, \cite{Iwa73mu}, \cite{NSW}, etc.)
Hence we would like to obtain a useful criterion for the stabilization of 
$|A (\mathbb{K}_n)|$ (or the $p$-rank of $A (\mathbb{K}_n)$).

In \cite{Hachi}, Hachimori gave fundamental results on 
``Iwasawa modules with restricted $p$-ramification''.
As an application of his results, Hachimori also gave
a criterion for the triviality of $X (\mathbb{K}_\infty)$ 
in a certain situation (\cite[Theorem 8.1]{Hachi}).
His result is considered as a re-interpretation
and a generalization of Fukuda-Komatsu's result (\cite{F-K}).
He also gave a precise criterion for the case where 
$\mathbb{K}$ is a pure cubic field and $p$
splits into two distinct primes in $\mathbb{K}$ 
(\cite[Example 8.3]{Hachi}).

In the present paper, we shall develop Hachimori's theory
for more general situations, and give criteria for the
finiteness of $|X (\mathbb{K})|$.
We also give criteria for the stabilization of the $p$-rank of $A (\mathbb{K}_n)$.
Moreover, we also study the cyclotomic $\mathbb{Z}_2$-extension of certain real quadratic 
fields by using our results.

\subsection{Assumptions on $K/F$ and $p$}\label{assumptions}
Throughout the present paper, an extension $K/F$ of algebraic number fields
and a prime number $p$ satisfy the following conditions:
\begin{itemize}
\item[(F1)] $F$ is totally real.
\item[(F2)] There is only one prime $\mathfrak{p}^0$ of $F$ which is lying above $p$.
\item[(F3)] $\mathfrak{p}^0$ is totally ramified in $F_\infty$.
\item[(F4)] Let $M_p (F) / F$ be the maximal abelian pro-$p$ extension
unramified outside $p$.
Then $M_p (F) = F_\infty$.
\item[(K1)] $K \cap F_\infty = F$.
\item[(K2)] There are exactly two primes $\mathfrak{p}$, $\mathfrak{p}'$
in $K$ lying above $p$.
\item[(K3)] Let $K_{\mathfrak{p}}$ (resp. $F_{\mathfrak{p}^0}$)
be the completion of $K$ (resp. $F$) at
$\mathfrak{p}$ (resp. $\mathfrak{p}^0$).
Then $K_{\mathfrak{p}} = F_{\mathfrak{p}^0}$. 
\item[(K4)] Both $\mathfrak{p}$ and $\mathfrak{p}'$ are totally ramified in $K_\infty$.
\end{itemize}
By these assumptions, for every non-negative integer $n$, 
we see that $A (F_n)$ is trivial and Leopoldt's conjecture holds for $F_n$.

The following are typical cases which satisfy these assumptions.
\begin{itemize}
\item $F = \mathbb{Q}$, $K$ is a quadratic
field, and $p$ splits in $K$.
\item $F = \mathbb{Q}$, $K$ is a (non-Galois) cubic  
field, and $p$ splits into two distinct primes in $K$ (cf. \cite{Hachi}).
\end{itemize}

\subsection{Notation}
We shall define notations will be used later.
In the following, $\mathbb{K}$ denotes a finite extension field
over $\mathbb{Q}$, and $\mathfrak{P}$ denotes a prime of $\mathbb{K}$ lying above $p$.
\begin{itemize}
\item $O_{\mathbb{K}}$ : the ring of integers of $\mathbb{K}$.
\item $E (\mathbb{K})$ : the group of units of $\mathbb{K}$.
\item $A (\mathbb{K})$ : the Sylow $p$-subgroup of the ideal class group.
\item $\mathbb{K}_\mathfrak{P}$ : the completion of $\mathbb{K}$ at $\mathfrak{P}$.
\item $L (\mathbb{K}) / \mathbb{K}$ : the maximal unramified abelian $p$-extension.
\item $X (\mathbb{K}) := \Gal (L (\mathbb{K}) / \mathbb{K})$
(recall that $X (\mathbb{K}) \cong A (\mathbb{K})$ by class field theory).
\end{itemize}
Same as above, we denote by 
$\mathbb{K}_\infty / \mathbb{K}$ the cyclotomic $\mathbb{Z}_p$-extension,
and $\mathbb{K}_n$ its $n$th layer ($n$ is a non-negative integer).
However, we use $\mathbb{B}_\infty$, $\mathbb{B}_n$ for the case where $\mathbb{K} = \mathbb{Q}$.
We define $L (\mathbb{K}_\infty)$, $X (\mathbb{K}_\infty)$ similar as above.
We also define the following:
\begin{itemize}
\item $M_\mathfrak{P} (\mathbb{K}_n) / \mathbb{K}_n$ : the maximal abelian pro-$p$ extension
unramified outside the primes lying above $\mathfrak{P}$.
\item $\mathfrak{X}_\mathfrak{P} (\mathbb{K}_n) :=
\Gal (M_\mathfrak{P} (\mathbb{K}_n) / \mathbb{K}_n)$.
\item $Z_\mathfrak{P} (\mathbb{K}_n) :=
\Gal (M_\mathfrak{P} (\mathbb{K}_n) / L (\mathbb{K}_n))$.
\end{itemize}
We also define $M_\mathfrak{P} (\mathbb{K}_\infty)$, 
$\mathfrak{X}_\mathfrak{P} (\mathbb{K}_\infty)$,
$Z_\mathfrak{P} (\mathbb{K}_\infty)$ similarly.

Here we shall define two subgroups of $A (\mathbb{K}_n)$ which will be used later.
Let $A (\mathbb{K}_n)^{\Gal (\mathbb{K}_n / \mathbb{K})}$ the 
$\Gal (\mathbb{K}_n / \mathbb{K})$-invariant subgroup of $\mathbb{K}$, 
and $\overline{A} (\mathbb{K}_n)^{\Gal (\mathbb{K}_n / \mathbb{K})}$ the 
subgroup consists of the classes which include an ideal invariant under the 
action of $\Gal (\mathbb{K}_n / \mathbb{K})$.
There are well known formulas of 
$|A (\mathbb{K}_n)^{\Gal (\mathbb{K}_n / \mathbb{K})}|$ 
and $|\overline{A} (\mathbb{K}_n)^{\Gal (\mathbb{K}_n / \mathbb{K})}|$ 
(see, e.g., \cite[p.307, Lemma 4.1]{Lang} and its proof).

For a $\mathbb{Z}_p$-module $\mathcal{X}$ which appears in the
present paper, we denote by $\rankp \mathcal{X}$ be the $p$-rank of $\mathcal{X}$
(that is, the dimension of $\mathcal{X} / p \mathcal{X}$ over $\mathbb{F}_p$).

Assume that $K/F$ and $p$ satisfy the assumptions stated in
Section \ref{assumptions}.
We define more several notations.
\begin{itemize}
\item $\mathfrak{p}_n$, $\mathfrak{p}'_n$ : the unique prime of $K_n$ lying above
$\mathfrak{p}$, $\mathfrak{p}'$, respectively.
\item $\mathfrak{p}^0_n$ : the unique prime of $F_n$ lying above
$\mathfrak{p}^0$.
\item $i_n$ : the map $A (K) \to A (K_n)$ induced from the
extension of ideals. ($i_0$ is the identity map.)
\item $L^\dagger (K_n) / K_n$ : the maximal unramified abelian $p$-extension 
such that $\mathfrak{p}'_n$ splits completely.
\item $L' (K_n) / K_n$ : the maximal unramified abelian $p$-extension 
such that $\mathfrak{p}_n$, $\mathfrak{p}'_n$ splits completely.
\item $X' (K_n) := \Gal (L' (K_n) / K_n)$.
\item $M'_\mathfrak{p} (K_n) / K_n$ : the maximal abelian pro-$p$ extension
unramified outside $\mathfrak{p}_n$ such that  $\mathfrak{p}'_n$ splits completely.
\item $\mathfrak{X}'_\mathfrak{p} (K_n) :=
\Gal (M'_\mathfrak{p} (K_n) / K_n)$.
\end{itemize}
We shall later show that $L^\dagger (K_n) = L' (K_n)$.

\subsection{Results}
We shall give general results in Section \ref{section_criteria}.
We also give the results limited to the case of real quadratic fields with $p=2$ 
in Section \ref{section_quadratic}.

Assume that $K/F$ and $p$ satisfies the assumptions stated in
Section \ref{assumptions}.
In Section \ref{section_criteria}, we give sufficient conditions
such that $X (K_\infty)$ is finite (i.e., $\lambda = \mu =0$),
or $X (K_\infty)$ is finitely generated
over $\mathbb{Z}_p$ (i.e., $\mu =0$).

In Section \ref{subsection_finiteness}, we shall consider the condition 
such that $|X (K_\infty)|$ is finite.
We state several criteria there.  
One of them is the following.

\begin{thmintro}[=Theorem \ref{finiteness_criterion1}]
Let $n$ be a positive integer.
Assume that $i_n (A (K))$ is trivial.
Then $M_{\mathfrak{p}} (K_n) = L (K_n)$ if and only if
$|X (K_\infty)| = |X (K_n)|$.
\end{thmintro}

When $K$ is a real quadratic field, the above and some of other 
results given in Section \ref{subsection_finiteness} relate to the results given 
in Fukuda-Taya \cite{F-T} (and Taya's thesis \cite{Taya}).
We shall mention such relationships in several places.

In Section \ref{subsection_Dn}, we shall show that 
under a certain limited situation, 
the equation $|\mathfrak{X}_\mathfrak{p} (K_n)| = |\mathfrak{X}_{\mathfrak{p}} (K)| 
\cdot |X' (K_n)|$ holds (Proposition \ref{Bn_Dn}).
We also give applications of this proposition.

In Section \ref{subsection_FT}, we shall consider the connection with 
the invariants defined in \cite{F-T}.

In Section \ref{subsection_rank}, 
we give several criteria for the stabilization of the $p$-rank of $X (K_n)$.
One of them is the following:

\begin{thmintro}[=Theorem \ref{rank_criterion1}]
Let $n$ be a positive integer.
Assume that $i_n (A (K))$ is trivial.
Then $\rankp \mathfrak{X}_\mathfrak{p} (K_n) = \rankp X (K_n)$ 
if and only if 
$\rankp X (K_\infty) = \rankp X (K_n)$.
\end{thmintro}

In Section \ref{section_quadratic}, we focus on a certain family of real quadratic fields.
In the following, let $q$ be an odd prime number 
and put $k = \mathbb{Q} (\sqrt{q})$.
We apply the results of Section \ref{section_criteria} 
as $F = \mathbb{Q}$, $K =k$, and $p=2$.

There are various studies of Greenberg's conjecture for 
real quadratic fields with $p=2$.
Recently, Fukuda-Komatsu-Kumakawa-Sasaki \cite{FKKS} showed that 
Greenberg's conjecture holds all $\mathbb{Q} (\sqrt{m})$ 
in the range $m < 1,000,000$ ($m$ is not limited to prime numbers).
We will give results not only on Greenberg's conjecture but also 
on the $2$-rank of $X(k_\infty)$.

In Section \ref{subsection_cyclic}, 
we shall give ``yet another'' equivalent condition such that $X (k_\infty)$
is cyclic.
Note that an equivalent condition on the cyclicity of $X (k_\infty)$
is already obtained by Mouhib-Movahhedi \cite{Mo-Mo}.
Recently, Mizusawa-Mouhib \cite{Mi-Mo} gave another equivalent condition.
Our condition is different from theirs.
Note that our condition is derived from more general results, and 
it seems that the same type result holds in other situations.

\begin{thmintro}[=Theorem \ref{theorem_cyclicity_real_quad}]
Assume that $q$ satisfy 
$q \equiv 1 \pmod{16}$ and $2^{\frac{q-1}{4}} \equiv 1 \pmod{q}$.
Then the following hold.
\begin{itemize}
\item[(1a)] If $| \mathfrak{X}_\mathfrak{p} (k) | =2$ and
$|A (k_1)| =2$, then $\ranktwo X (k_\infty) >1$.
\item[(1b)] If $| \mathfrak{X}_\mathfrak{p} (k) | =2$ and
$|A (k_1)| >2$, then $\ranktwo X (k_\infty) =1$.
\item[(2a)] If $| \mathfrak{X}_\mathfrak{p} (k) | > 2$ and
$|A (k_1)| =2$, then $\ranktwo X (k_\infty) =1$.
\item[(2b)] If $| \mathfrak{X}_\mathfrak{p} (k) | > 2$ and
$|A (k_1)| >2$, then $\ranktwo X (k_\infty) >1$.
\end{itemize}
\end{thmintro}

We note that one can determine whether $|A (k_1)| =2$ or 
$|A (k_1)| > 2$ by observing the ideal class group of 
$\mathbb{Q} (\sqrt{2q})$ (see Remark \ref{remark_2q}).
Our condition is also used to prove the results appeared later.

We also give several results for Greenberg's conjecture in Section \ref{subsection_GC}.

\begin{thmintro}[=Corollary \ref{cor_full_rank}]
Assume that $q$ satisfies  
$q \equiv 1 \pmod{16}$ and $2^{\frac{q-1}{4}} \equiv 1 \pmod{q}$.
Assume also that $|A (k_1)|=2$.
We put $e = v_2 (q-1) -2$, where $v_2 ( \cdot )$ is the normalized additive 
$2$-adic valuation.
If $\ranktwo A(k_e) = 2^{e} - 1$, then $|X (k_\infty)| = |X (k_e)|$.
\end{thmintro}

\begin{thmintro}[=Theorem \ref{theorem_second_layer}]\label{intro_theorem_second_layer}
Assume that $q$ satisfies 
$q \equiv 1 \pmod{16}$ and $2^{\frac{q-1}{4}} \equiv 1 \pmod{q}$.
Assume also that $|A (k_1)|=2$, and
\[ A (k_2) \cong \mathbb{Z} / 2 \mathbb{Z} \oplus \mathbb{Z} / 4 \mathbb{Z} \]
as an abelian group.
Then, $|X (k_\infty)|$ is finite.
\end{thmintro}

For the above Theorem \ref{intro_theorem_second_layer}, 
a slightly stronger result is already known 
when $q \equiv 17 \pmod{32}$ (see Remark \ref{rem_previous_results}).
However, our method is different from previous ones.
Finally, we propose several conjectures based on the computation of $k_2$.

Although our study in Section \ref{section_quadratic} 
is limited to the case of real quadratic fields, 
the results given in Section \ref{section_criteria} are applicable for 
more general situations (e.g., complex cubic fields).
Applying these results to other cases is the future subject.

\section{General criteria for the stabilization of the order 
or the $p$-rank}\label{section_criteria}

Let $K/F$ an extension of algebraic number fields satisfying
the assumptions on $K/F$ and $p$ stated in Section \ref{assumptions}.

\subsection{Preparatory results}
We shall recall or prove several preliminary results.

\begin{definition}
Since $K_\mathfrak{p} = F_{\mathfrak{p}^0}$, we see that
$(K_n)_{\mathfrak{p}_n} = (F_n)_{\mathfrak{p}^0_n}$ for all $n$.
We denote $\mathcal{U}_n$ by the group of principal units
of $(K_n)_{\mathfrak{p}_n} = (F_n)_{\mathfrak{p}^0_n}$.
Let $E^1 (F_n)$ (resp.  $E^1 (K_n)$) be the subgroup of
$E (F_n)$ (resp. $E (K_n)$) which is congruent to 1 modulo
$\mathfrak{p}^0_n$ (resp. $\mathfrak{p}_n$).
We also denote by $\mathcal{E} (F_n)$ (resp. $\mathcal{E} (K_n)$)
the closure of $E^1 (F_n)$ (resp. $E^1 (K_n)$) in $\mathcal{U}_n$.
We note that $\mathcal{E} (F_n)$ is a subgroup
of $\mathcal{E} (K_n)$.
\end{definition}

By class field theory, we see that $Z_{\mathfrak{p}} (K_n) = 
\Gal (M_\mathfrak{p} (K_n)/L (K_n))
\cong \mathcal{U}_n / \mathcal{E}(K_n)$ (see \cite{Hachi}).

\begin{lem}\label{lem_cyclic_F_n}
Let $n$ be a non-negative integer.

\smallskip

\noindent (1)
$\mathcal{U}_n / \mathcal{E} (F_n) \cong \Gal (F_\infty / F_n)$.

\smallskip

\noindent (2)
$\mathcal{U}_n / \mathcal{E} (K_n)$ is a procyclic pro-$p$ group.
\end{lem}

\begin{proof}
It is well known that $\Gal (M_p (F_n)/ F_n)
\cong \mathcal{U}_n / \mathcal{E}(F_n)$ 
(recall that the class number of $F_n$ is prime to $p$ under our assumptions).
Since $M_p (F_\infty) = F_\infty$, there is no non-trivial
abelian $p$-extension over $F_\infty$ unramified outside $p$.
Hence, $M_p (F_n) = F_\infty$.
The assertion (1) follows from this.

By (1), we see that $\mathcal{U}_n / \mathcal{E} (F_n)$ is
a procyclic pro-$p$ group.
Since $\mathcal{U}_n / \mathcal{E} (K_n)$ is a quotient of
$\mathcal{U}_n / \mathcal{E} (F_n)$, we obtain the assertion of (2)
(see also \cite{Hachi}).
\end{proof}

We note that $\mathcal{U}_n / \mathcal{E} (K_n)$ can be infinite
(e.g., the case where $K$ is an imaginary quadratic field).
In this paper, we mainly consider the case where
$\mathcal{U}_n / \mathcal{E} (K_n)$ is finite for all $n$.

Recall that $Z_{\mathfrak{p}} (K_\infty) = 
\Gal (M_\mathfrak{p} (K_\infty) / L(K_\infty)) \cong
\varprojlim \mathcal{U}_n / \mathcal{E} (K_n)$,
where the projective limit is taken with respect to the norm maps
(\cite{Hachi}).
The following is a crucial result of our study.
This is a slight generalization of Lemma 8.2 of Hachimori \cite{Hachi}.
(Fujii \cite[Lemma 2]{Fujii} also showed the same assertion for imaginary quadratic fields with odd $p$ 
by using slightly different method from Hachimori's.
We note that a similar result in a certain situation 
also can be found in Ozaki \cite{Oza_proc}.)

\begin{prop}\label{Iwasawa_modules_coincidence}
$\varprojlim \mathcal{U}_n / \mathcal{E} (K_n)$ is trivial.
That is, $M_{\mathfrak{p}} (K_\infty) = L (K_\infty)$ holds.
\end{prop}

\begin{proof}
As mentioned in Lemma \ref{lem_cyclic_F_n},
$\mathcal{U}_n / \mathcal{E} (K_n)$ is a quotient of
$\mathcal{U}_n / \mathcal{E} (F_n)$.
Hence the assertion follows because
$\varprojlim \mathcal{U}_n / \mathcal{E} (F_n)$ is trivial.
\end{proof}

From the above result, we see that $\mathfrak{X}_\mathfrak{p} (K_\infty)$
is a finitely generated torsion $\mathbb{Z}_p [[ \Gal (K_\infty /K) ]]$-module
(as well as $X (K_\infty)$).
We note that $M'_\mathfrak{p} (K_\infty) = L^\dagger (K_\infty)$ also holds.

\begin{definition}\label{definition_X}
We often write the same object
$\mathfrak{X}_\mathfrak{p} (K_\infty) = X (K_\infty)$ simply as $X$.
\end{definition}

\begin{definition}\label{definition_lambda}
We fix a topological generator $\gamma$ of $\Gal (K_\infty /K)$.
We put $\Lambda = \mathbb{Z}_p [[T]]$.
Take an isomorphism
$\mathbb{Z}_p [[ \Gal (K_\infty /K) ]] \to \Lambda$ with  
$\gamma \mapsto 1+T$,
and we see that $X (K_\infty), \mathfrak{X}_\mathfrak{p} (K_\infty)$ as
$\Lambda$-modules via this isomorphism.
For every non-negative integer $n$, we define the following
elements of $\Lambda$:
\[ \omega_n = (1+T)^{p^n} -1, \qquad \nu_n = \omega_n / T. \]
\end{definition}

The following is shown in Hachimori's paper.

\begin{thm}[Hachimori \cite{Hachi}]\label{quotient_restricted}
Let $n$ be a non-negative integer.

\smallskip

\noindent (1)
$\mathfrak{X}_{\mathfrak{p}} (K_n) \cong
\mathfrak{X}_{\mathfrak{p}} (K_\infty) / \omega_n  
\mathfrak{X}_{\mathfrak{p}} (K_\infty)$.

\smallskip

\noindent (2) $\mathfrak{X}'_{\mathfrak{p}} (K_n) \cong
\mathfrak{X}'_{\mathfrak{p}} (K_\infty) / \omega_n  
\mathfrak{X}'_{\mathfrak{p}} (K_\infty)$.
\end{thm}

From the above isomorphisms, we can obtain the following result.

\begin{cor}\label{cor_Hachimori}
(1) $M_{\mathfrak{p}} (K) = K$ if and only if
$X (k_\infty)$ is trivial.

\smallskip

\noindent (2) $M'_{\mathfrak{p}} (K) = K$ if and only if
$X' (k_\infty)$ is trivial.
\end{cor}

For the above corollary, Hachimori \cite[Theorem 8.1]{Hachi} showed
the assertion (1) for the case where $F = \mathbb{Q}$.
The above corollary gives slightly generalized results including the
case of $X' (k_\infty)$.
However, these can be shown by using the same argument as given
in \cite{Hachi}, hence we omit the proof.

We recall that $L^\dagger (K_n)$ 
(resp. $L' (K_n)$) is the maximal unramified abelian $p$-extension
such that only $\mathfrak{p}'_n$ splits completely 
(resp. both $\mathfrak{p}_n$, $\mathfrak{p}'_n$ split completely).
However, we can see that $L' (K_n) = L^\dagger (K_n)$.
This follows from the below result.

\begin{lem}\label{D_n_generator}
Fix a non-negative integer $n$.
Let $c(\mathfrak{p}'_n)$ (resp. $c(\mathfrak{p}'_n)$) be the ideal class
of $K_n$ which includes $\mathfrak{p}_n$ (resp. $\mathfrak{p}'_n$).
Then
\[ A (K_n) \cap \langle c(\mathfrak{p}_n), c(\mathfrak{p}'_n) \rangle =
A (K_n) \cap \langle c(\mathfrak{p}'_n) \rangle. \]
\end{lem}

\begin{proof}
Recall that $\mathfrak{p}^0_n$ is the unique prime of $F_n$ lying above $p$.
Then, $\mathfrak{p}^0_n$ is decomposed in $K_n$ as follows:
\[ \mathfrak{p}^0_n O_{K_n} = \mathfrak{p}_n (\mathfrak{p}'_n)^{d}, \]
where $d$ is a positive integer.
Let $c(\mathfrak{p}^0_n O_{K_n})$ be the ideal class of $K_n$
which includes $\mathfrak{p}^0_n O_{K_n}$.
Since the class number of $F_n$ is prime to $p$,
the order of $c(\mathfrak{p}^0_n O_{K_n})$ is also prime to $p$.
Let $e$ be the largest divisor of the class number of $K_n$ 
which is prime to $p$, 
then $c(\mathfrak{p}^0_n O_{K_n})^e$ is trivial.
From these results, we see that 
\[ c(\mathfrak{p}_n)^{e} = c (\mathfrak{p}'_n)^{-d e} \]
in the ideal class group of $K_n$.
Since $c(\mathfrak{p}_n)^{e}$ and $c(\mathfrak{p}'_n)^{e}$
generates $A (K_n) \cap \langle c(\mathfrak{p}_n), c(\mathfrak{p}'_n) \rangle$,
the above equality shows our assertion.
\end{proof}

Hence, we also see that $L' (K_\infty)$ is the inertia field of 
$M_\mathfrak{p} (K_\infty) / K_\infty$ for the prime lying above $\mathfrak{p}'$.

\begin{definition}\label{definition_D_n}
We put
\[ D (K_n) = A (K_n) \cap \langle c(\mathfrak{p}_n), c(\mathfrak{p}'_n) \rangle
(= A (K_n) \cap \langle c(\mathfrak{p}'_n) \rangle) \]
and $A' (K_n) = A (K_n)/ D (K_n)$.
\end{definition}

The following is widely known.

\begin{thm}[cf. \cite{Iwa73}]\label{quotient_unramified}
Let $n$ be a non-negative integer.

\smallskip

\noindent (1) $X (K_n) \cong
X (K_\infty) / \nu_n Y$, where $Y = \Gal (L (K_\infty) / L (K) K_\infty)$.

\smallskip

\noindent (2) $X' (K_n) \cong
X' (K_\infty) / \nu_n Y'$, where $Y' = \Gal (L' (K_\infty) / L' (K) K_\infty)$.
\end{thm}

We note that $| \mathfrak{X}_\mathfrak{p} (K) |$ relates to the 
order of the $\Gal (K_n / K)$-invariant subgroup of $A (K_n)$.

\begin{lem}\label{finiteness_B_n}
Assume that $| \mathfrak{X}_\mathfrak{p} (K) |$ is finite.
We put $G_n = \Gal (K_n/K)$.
If $n$ is sufficiently large, then
$| A (K_n)^{G_n}| = | \mathfrak{X}_\mathfrak{p} (K) |$.
\end{lem}

\begin{proof}
In this proof, we denote by $\mathcal{X}_{G_n}$ the
$G_n$-coinvariant quotient of $\mathcal{X}$.
First we note that
\[ |A (K_n)^{G_n} | = |A (K_n)_{G_n} | = |X (K_n)_{G_n}|. \]

Recall that there are only two primes $\mathfrak{p}, \mathfrak{p}'$
are ramified in $K_n /K$, and $\mathfrak{p}'$ is totally ramified.
From this, we can see that $\mathfrak{X}_\mathfrak{p} (K_n)_{G_n}$ is
isomorphic to $\mathfrak{X}_\mathfrak{p} (K)$.

By Proposition \ref{Iwasawa_modules_coincidence},
$M_\mathfrak{p} (K) K_n /K_n$ must be unramified
if $n$ is sufficiently large.
Since $X (K_n)_{G_n}$ is a quotient of
$\mathfrak{X}_\mathfrak{p} (K_n)_{G_n}$,
we see that
\[ |X (K_n)_{G_n}| = |\mathfrak{X}_\mathfrak{p} (K_n)_{G_n}| \]
if $n$ is sufficiently large.

By combining the above equations, the assertion follows.
\end{proof}

\subsection{Criteria for the finiteness of $|X (K_\infty)|$}\label{subsection_finiteness}
In this subsection, we shall show several criteria for 
$X (K_\infty)$ being finite.

\begin{lem}\label{if_part_finite}
Assume that $M_{\mathfrak{p}} (K_n) \neq L (K_n)$
for a non-negative integer $n$.
Then $|X (K_\infty)| > | X (K_n) |$.
\end{lem}

\begin{proof}
By the assumption, we obtain the following:
\[ |X (K_n)| < | \mathfrak{X}_\mathfrak{p} (K_n) |
\leq | \mathfrak{X}_\mathfrak{p} (K_\infty) |
= | X (K_\infty) | . \]
\end{proof}

Hence, if $|X (K_\infty)|$ is finite, then
$M_{\mathfrak{p}} (K_n) = L (K_n)$ is satisfied for all
sufficiently large $n$.
In the following, we shall consider the converse of this.

\begin{thm}\label{finiteness_criterion1}
Let $n$ be a non-negative integer.
Assume that $i_n (A (K))$ is trivial.
Then $M_{\mathfrak{p}} (K_n) = L (K_n)$ if and only if
$|X (K_\infty)| = |X (K_n)|$.
\end{thm}

\begin{proof}
The ``only if'' part follows from Lemma \ref{if_part_finite}.
Hence we shall show the ``if'' part.

We recall that $X$ denotes the same object
$\mathfrak{X}_\mathfrak{p} (K_\infty) = X (K_\infty)$.
We note that $i_n$ is also expressed as the following from.
\[ i_n : X/ Y \to X / \nu_n Y, \quad x \to \nu_n x \]
(see, e.g., \cite{Iwa73}).
Hence, if $i_n (A (K))$ is trivial, then $\nu_n X$ is contained in
$\nu_n Y$.
This implies that $\nu_n X = \nu_n Y$, and hence
$X (K_n) \cong X /\nu_n X$.

Assume that $M_{\mathfrak{p}} (K_n) = L (K_n)$.
By Theorem \ref{quotient_restricted},
we obtain that  
\[ \mathfrak{X}_\mathfrak{p} (K_n) \cong X / \omega_n X. \]
Hence, we see that $\nu_n X = \omega_n X$.
This also implies that $\nu_n X = T (\nu_n X)$.
From this, by using topological Nakayama's lemma,
we can conclude that $\nu_n X$ is trivial.
Hence $X = X (K_\infty) \cong X (K_n)$.
\end{proof}

In particular, the above theorem is applicable for the 
case where $A (K)$ is trivial.

\begin{thm}\label{split_finiteness_criterion}
Let $n$ be a non-negative integer.
We denote by $i'_n : A' (K) \to A' (K_n)$ the map induced from 
the extension of ideals.
Assume that $i'_n (A' (K))$ is trivial.
Then $M'_{\mathfrak{p}} (K_n) = L' (K_n)$ if and only if
$X' (K_\infty) \cong X' (K_n)$.
\end{thm}

\begin{proof}
We can also show the assertion by using the argument given in
the proof of Theorem \ref{finiteness_criterion1}.
\end{proof}

\begin{lem}\label{split_finiteness_implies}
Assume that $|\mathfrak{X}_\mathfrak{p} (K)|$ is finite.
If $|X' (K_\infty)|$ is finite, then $|X (K_\infty)|$ is also finite.
\end{lem}

\begin{proof}
By Lemma \ref{finiteness_B_n}, we see that $| A (K_n)^{\Gal (K_n /k)}|$
is bounded with respect to $n$.
Since, $D (K_n) \subset A (K_n)^{\Gal (K_n /k)}$, we also see that
$| D(K_n)|$ is bounded.
Hence, if $|A' (K_n)|$ is bounded, then $|A (K_n)|$ is also bounded.
The assertion follows from this.
\end{proof}

\begin{rem}\label{Fukuda_analog}
Since $\mathfrak{X}_\mathfrak{p} (K_n) \cong
\mathfrak{X}_\mathfrak{p} (K_\infty) / \omega_n \mathfrak{X}_\mathfrak{p} (K_\infty)$,
we can obtain an analog of Fukuda's theorem \cite[Theorem 1(1)]{Fuku94}.
That is, if $|\mathfrak{X}_\mathfrak{p} (K_n)| =
|\mathfrak{X}_\mathfrak{p} (K_{n+1})|$, then
$|\mathfrak{X}_\mathfrak{p} (K_\infty)| = |\mathfrak{X}_\mathfrak{p} (K_n)|$
(hence $|A (K_m)| = |\mathfrak{X}_\mathfrak{p} (K_n)|$ for all
sufficiently large $m \geq n$).
We can show this by imitating the original proof.
(Of course, a similar result for $\mathfrak{X}'_\mathfrak{p} (K_\infty)$
is also can be obtained.)
We note that the case that
\[ |A (K_n)| < |A (K_{n+1})| < |\mathfrak{X}_\mathfrak{p} (K_n)| =
|\mathfrak{X}_\mathfrak{p} (K_{n+1})| \]
may occur, hence this criterion seems more useful depending on the situation.
However, there are examples such that both 
$|\mathfrak{X}_\mathfrak{p} (K_n)| < |\mathfrak{X}_\mathfrak{p} (K_{n+1})|$
and
$|\mathfrak{X}'_\mathfrak{p} (K_n)| < |\mathfrak{X}'_\mathfrak{p} (K_{n+1})|$
is satisfied but Theorem \ref{finiteness_criterion1} is applicable for $K_{n+1}$ 
(see Example \ref{example_rank_7}).
Moreover, we later give a result on the relationship between the 
stabilization of $|\mathfrak{X}_\mathfrak{p} (K_n)|$ and 
that of $|X' (K_n)|$ under a certain situation 
(Corollary \ref{Consequence_Fukuda_analog}).
\end{rem}

We shall show that $|Z_\mathfrak{p} (K_n)|$ does not increase 
(if it is finite).

\begin{prop}\label{inertia_subgroup_order}
Let $n$ be a non-negative integer.
Assume that $|Z_\mathfrak{p} (K_n)|$ is finite.
Then $|Z_\mathfrak{p} (K_{n+1})|$ is also finite,
and
\[ |Z_\mathfrak{p} (K_{n+1})| \leq
|Z_\mathfrak{p} (K_n)|. \]
In particular, if $|Z_\mathfrak{p} (K_n)|$ is trivial,
then $|Z_\mathfrak{p} (K_{n+1})|$ is trivial.
\end{prop}

\begin{proof}
Recall that $|Z_\mathfrak{p} (K_{n})| = |\mathcal{U}_n / \mathcal{E} (K_n)|$
and $|Z_\mathfrak{p} (K_{n+1})| = |\mathcal{U}_{n+1} / \mathcal{E} (K_{n+1})|$.
To prove the assertion, we shall compare
$|\mathcal{U}_n / \mathcal{E} (K_n)|$ and $|\mathcal{U}_{n+1} / \mathcal{E} (K_{n+1})|$.

Consider the following commutative diagram:
\[ \begin{CD}
\mathcal{U}_{n+1} / \mathcal{E} (F_{n+1}) @>{\sim}>> \Gal (F_\infty / F_{n+1}) \\
@VV{\text{norm}}V @VV{\text{restriction}}V \\
\mathcal{U}_{n} / \mathcal{E} (F_{n}) @>{\sim}>> \Gal (F_\infty / F_{n}),
\end{CD} \]
We note that the right vertical map is actually the natural injection,
and the order of its cokernel is exactly $p$.
Then left vertical map is also injective, and
its cokernel also has order exactly $p$.

Next, we consider the following commutative diagram with exact rows:
\[ \begin{CD}
\mathcal{E} (K_{n+1}) @>>> \mathcal{E} (K_{n+1})/ \mathcal{E} (F_{n+1}) @>>> 0 \\
@VV{\text{norm}}V @VVV \\
\mathcal{E} (K_{n}) @>>> \mathcal{E} (K_{n})/ \mathcal{E} (F_{n}) @>>> 0
\end{CD}. \]
We note the following two facts:
the exponent of the cokernel of the left vertical map is at most $p$
(since $K_{n+1} /K_n$ is a cyclic extension of degree $p$),
and $\mathcal{E} (K_{n})/ \mathcal{E} (F_{n})$ is a
cyclic $\mathbb{Z}_p$-module (since $\mathcal{U}_n / \mathcal{E} (F_{n})$
is cyclic).
From these facts, we see that the cokernel of the right vertical
map has order $1$ or $p$.

Finally, we consider the following commutative diagram with exact rows
and columns:
\[ \begin{CD}
@. @. 0 @>>> \mathscr{A} @. \\
@. @. @VVV @VVV @. \\
0 @>>>
\mathcal{E} (K_{n+1})/ \mathcal{E} (F_{n+1}) @>>>
\mathcal{U}_{n+1} / \mathcal{E} (F_{n+1}) @>>>
\mathcal{U}_{n+1} / \mathcal{E} (K_{n+1}) @>>> 0  \\
@. @VVV @VVV @VVV @. \\
0 @>>> \mathcal{E} (K_{n})/ \mathcal{E} (F_{n}) @>>>
\mathcal{U}_{n} / \mathcal{E} (F_{n}) @>>>
\mathcal{U}_{n} / \mathcal{E} (K_{n}) @>>> 0 \\
@. @VVV @VVV @VVV @. \\
@. \mathscr{B} @>>> \mathscr{C} @>>> \mathscr{D} @>>>0
\end{CD}, \]
where $\mathscr{A}$, $\mathscr{B}$, $\mathscr{C}$, $\mathscr{D}$ 
are the kernel or the cokernel of the vertical maps.
By the facts stated in the second preceding paragraph,
the central vertical map is injective and $|\mathscr{C}| = p$.
We also note that $|\mathscr{B}|$ is $1$ or $p$ by the fact 
stated in the preceding paragraph.
By applying the snake lemma, we obtain the exact sequence
\[ 0 \to \mathscr{A} \to \mathscr{B} \to \mathscr{C} \to \mathscr{D} \to 0. \]
Since $|\mathscr{C}| =p$ and $|\mathscr{B}| \leq p$, the possibility of the
orders of $\mathscr{A}$, $\mathscr{B}$, $\mathscr{C}$, $\mathscr{D}$ are the following.

\[ \begin{array}{cccc}
|\mathscr{A}| & |\mathscr{B}| & |\mathscr{C}| & |\mathscr{D}| \\
\hline
p & p & p & p \\
1 & p & p & 1 \\
1 & 1 & p & p
\end{array} \]

In any case, we can see that $|\mathcal{U}_{n+1} / \mathcal{E} (K_{n+1})|$
must be equal to $|\mathcal{U}_{n} / \mathcal{E} (K_{n})|$
or $|\mathcal{U}_{n} / \mathcal{E} (K_{n})| / p$.
We then obtain the assertion.
\end{proof}

\begin{rem}
Combining Theorem \ref{finiteness_criterion1} and the above
Proposition \ref{inertia_subgroup_order}, we see that
if both of the following conditions
\begin{itemize}
\item $i_n (A (K))$ is trivial for some $n$, and 
\item $M_{\mathfrak{p}} (K_m) = L (K_m)$ for some $m$
\end{itemize}
are satisfied, then $|X (K_\infty)|$ is finite.
\end{rem}

At the end of this subsection, we shall state another criterion 
for the finiteness of $|X (K_\infty)|$.
The following seems useful 
if we know that $\rankp X (K_\infty) = 1$ in advance.
(We will give cyclicity criteria in Section \ref{subsection_rank}.)
Note that one can find the essence of this result in the proof of  
\cite[Theorem 1]{F-T} (see also Section \ref{subsection_FT}).
However, we give a simple proof by using $M_\mathfrak{p} (K_n) /K_n$, 
and we also mention the stabilization of $|\mathfrak{X}'_\mathfrak{p} (K_n)|$.

\begin{prop}\label{cyclic_finiteness_criterion}
Assume that $|\mathfrak{X}_\mathfrak{p} (K)|$ is finite, and
$X (K_\infty)$ is cyclic as a $\mathbb{Z}_p$-module.
If $M_\mathfrak{p} (K_n) \neq M'_\mathfrak{p} (K_n)$ for some
$n \geq 0$, then $|\mathfrak{X}'_\mathfrak{p} (K_m)|
= |\mathfrak{X}'_\mathfrak{p} (K_n)|$ holds for all $m > n$.
Moreover, $X (K_\infty)$ is finite.
\end{prop}

\begin{proof}
By the assumption, we see that
$\mathfrak{X}_\mathfrak{p} (K_m)$ is a finite cyclic group.
The decomposition subgroup of $\Gal (M_\mathfrak{p} (K_n) K_m / K_m)$ for 
$\mathfrak{p}'_m$ is $M'_\mathfrak{p} (K_n) K_m$.
Hence, if $M_\mathfrak{p} (K_n) / M'_\mathfrak{p} (K_n)$ is a
non-trivial extension, then the decomposition subgroup of
$\Gal (M_\mathfrak{p} (K_m) / K_m)$ for 
$\mathfrak{p}'_m$ is $M'_\mathfrak{p} (K_n) K_m$.
This implies that
$M'_\mathfrak{p} (K_m) =M'_\mathfrak{p} (K_n) K_m$, and
hence $|\mathfrak{X}'_\mathfrak{p} (K_m)|
= |\mathfrak{X}'_\mathfrak{p} (K_n)|$.

By the above result, we see that $|X' (K_\infty)|$ is finite.
(In fact, $|X' (K_\infty)| = |\mathfrak{X}'_\mathfrak{p} (K_n)|$.)
Then the finiteness of $X (K_\infty)$ also follows from
Lemma \ref{split_finiteness_implies}.
\end{proof}

\subsection{Relationship with $D (K_n)$}\label{subsection_Dn}

In the study of Greenberg's conjecture, the object like $D (K_n)$ 
often plays an important role (e.g., \cite[Theorem 2]{Gre76}).
We shall give a formula which describes a relationship between 
$|\mathcal{U}_n / \mathcal{E} (K_n)| (=|Z_\mathfrak{p} (K_n)|)$ 
and $|D (K_n)|$.
This is obtained inspired by the study of Fukuda-Taya \cite{F-T}.
We denote by $\rankZ E (F)$, $\rankZ E (K)$ the free rank of $E (F)$, $E (K)$,
respectively.

\begin{prop}\label{Bn_Dn}
Let $n$ be a non-negative integer.
Assume that all of the following conditions are satisfied:
\begin{itemize}
\item $| \mathfrak{X}_\mathfrak{p} (K) |$ is finite,
\item $i_n (A(K)) \subset D (K_n)$, and
\item $\rankZ E(K) = \rankZ E (F) + 1$.
\end{itemize}
Then, the following equality holds.
\begin{equation}\label{eq_Bn_Dn}
|\mathcal{U}_n / \mathcal{E} (K_n)| \cdot | D(K_n)| =
| \mathfrak{X}_\mathfrak{p} (K) |.
\end{equation}
Moreover, $| \mathfrak{X}_\mathfrak{p} (K_m) | = 
| \mathfrak{X}_\mathfrak{p} (K) | \cdot | X' (K_m) |$ holds for all $m \geq n$.
\end{prop}

\begin{proof}
If $A (K) = D (K)$, then 
\[ | \mathfrak{X}_\mathfrak{p} (K)| = 
|\mathcal{U}_0 / \mathcal{E} (K)| \cdot 
| A (K) | = 
|\mathcal{U}_0 / \mathcal{E} (K)| \cdot 
| D (K) |. \]
Hence, (\ref{eq_Bn_Dn}) holds for the case where $n = 0$.
In the remaining part of this proof, we assume that $n >0$.

First, we shall show the following claim.

\medskip

\noindent \underline{Claim 1.} The norm map $\mathcal{E} (F_n) \to \mathcal{E} (F)$
is surjective.

\medskip

Since the class number of $F_n$ is prime to $p$ and only one prime
ramifies in $F_n / F$, we can see that the norm map $E (F_n) \to E (F)$
is surjective (e.g., use a well known formula of 
$\overline{A} (F_n)^{\Gal (F_n/F)}$, which is also mentioned below).
Since $(E (F_n) : E^1 (F_n) )$ is prime to $p$,
we also see that the norm map $E^1 (F_n) \to E^1 (F)$
is surjective.
Hence the induced map $E^1 (F_n) \otimes_\mathbb{Z} \mathbb{Z}_p \to
E^1 (F) \otimes_\mathbb{Z} \mathbb{Z}_p$ is also surjective.
We note that the following diagram is commutative:
\[ \begin{CD}
E^1 (F_n) \otimes_\mathbb{Z} \mathbb{Z}_p @>{\sim}>> \mathcal{E} (F_n) \\
@VVV @VVV \\
E^1 (F) \otimes_\mathbb{Z} \mathbb{Z}_p @>{\sim}>> \mathcal{E} (F),
\end{CD} \]
here the vertical maps are the above stated ones, and 
the horizontal maps are induced from the embedding (cf. \cite[Chapter X]{NSW}).
We also note that the rows are surjective
(in fact, they are isomorphisms because Leopoldt's conjecture holds
for $F_n$).
Then the right vertical map is surjective, and the claim follows.

By the assumption and Proposition \ref{inertia_subgroup_order},
we see that
$|\mathfrak{X}_\mathfrak{p} (K_n)|$ is finite.
Hence $\mathcal{U}_n / \mathcal{E} (K_n)$ is finite, and
this implies that $\mathcal{E} (K_n) / \mathcal{E} (F_n)$ is
not trivial
(it is actually a free $\mathbb{Z}_p$-module of rank $1$).

We denote by $\Lnorm_{n,0} : \mathcal{E} (K_n) \to \mathcal{E} (K)$
the norm map.
We shall consider the following commutative diagram with exact rows:
\[ \begin{CD}
0 @>>> \mathcal{E} (F_n) @>>> \mathcal{E} (K_n) @>>>
\mathcal{E} (K_n)/ \mathcal{E} (F_n) @>>> 0 \\
@. @VVV @VV{\Lnorm_{n,0}}V @VV{\psi}V @. \\
0 @>>> \mathcal{E} (F) @>>> \mathcal{E} (K) @>>>
\mathcal{E} (K)/ \mathcal{E} (F) @>>> 0.
\end{CD} \]
We denote by $\mathscr{B}$ the cokernel of $\psi$. 
By Claim 1, $\mathscr{B}$ is isomorphic to
$\mathcal{E} (K) / \Lnorm_{n,0} \mathcal{E} (K_n)$.

Next, we consider the following commutative diagram with exact rows:
\[ \begin{CD}
0 @>>> \mathcal{E} (K_n) /\mathcal{E} (F_n) @>>> \mathcal{U}_n / \mathcal{E} (F_n) @>>>
\mathcal{U}_n / \mathcal{E} (K_n) @>>> 0 \\
@. @VV{\psi}V @VV{\varphi}V @VVV @. \\
0 @>>> \mathcal{E} (K) /\mathcal{E} (F) @>>> \mathcal{U}_0 / \mathcal{E} (F) @>>>
\mathcal{U}_0 / \mathcal{E} (K) @>>> 0,
\end{CD} \]
where the vertical maps are induced from the norm map.
The central vertical map $\varphi$ is injective, and
its cokernel (denote by it $\mathscr{C}$) 
is a cyclic group with order $p^n$
(these can be shown by using the same argument which is given in
the proof of Proposition \ref{inertia_subgroup_order}).
Let $\mathscr{A}$ (resp. $\mathscr{D}$) be the kernel (resp. cokernel) 
of the right vertical map. 
Then, by applying the snake lemma, we obtain the exact sequence:
\[ 0 \to \mathscr{A} \to \mathscr{B} \to \mathscr{C} \to \mathscr{D} 
\to 0. \]
Since $|\mathscr{C}| = p^n$, we obtain the equality
\[ |\mathscr{A}| \cdot p^n = |\mathscr{B}| \cdot |\mathscr{D}|. \]
Moreover, by using the facts that $|\mathscr{B}| =
|\mathcal{E} (K) / \Lnorm_{n,0} \mathcal{E} (K_n)|$ and
\[ |\mathscr{A}| \cdot |\mathcal{U}_0 / \mathcal{E} (K)| =
|\mathcal{U}_n / \mathcal{E} (K_n)| \cdot |\mathscr{D}| \]
(this follows from the right vertical map),
we obtain the equality
\begin{equation}\label{equation_local_units}
|\mathcal{E} (K) / \Lnorm_{n,0} \mathcal{E} (K_n)| =
\dfrac{p^n \cdot |\mathcal{U}_n / \mathcal{E} (K_n)|}{|\mathcal{U}_0 / \mathcal{E} (K_0)|}.
\end{equation}

The following argument is essentially same to that of given in
\cite{F-T} and \cite{Taya}.
Let $\Gnorm_{n,0} : E (K_n) \to E (K)$ be the norm map.
We put $G_n = \Gal (K_n / K)$.
Recall that $\overline{A} (K_n)^{G_n}$ is the subgroup of $A (K_n)$
consists on the ideal classes
which include an ideal invariant under the action of $G_n$.
By a well known formula of $|\overline{A} (K_n)^{G_n}|$, we see
\[ |\overline{A} (K_n)^{G_n}| = |A (K)| \cdot
\dfrac{p^n}{(E (K) : \Gnorm_{n,0} E (K_n))}. \]
Here, we note that
\[ \overline{A} (K_n)^{G_n} = i_n (A (K)) D (K_n). \]
Hence, from our assumption that $i_n (A (K)) \subset D (K_n)$,
we see that $\overline{A} (K_n)^{G_n} = D (K_n)$.
By combining these facts, we obtain the equality
\begin{equation}\label{equation_global_units}
|E (K) / \mskip -3mu \Gnorm_{n,0} E (K_n)| =
\dfrac{p^n \cdot |A (K)|}{|D (K_n)|}.
\end{equation}

We shall show the following claim.

\medskip

\noindent \underline{Claim 2.} Under the assumption that
$\rankZ E(K) = \rankZ E (F) + 1$, we see that \\
$|\mathcal{E} (K) / \Lnorm_{n,0} \mathcal{E} (K_n)|
= |E (K) / \Gnorm_{n,0} E (K_n)|$.

\medskip

Since $|E (K_n) / E^1 (K_n)|$ is prime to $p$ for every $n \geq 0$,
we see that $|E (K) / \Gnorm_{n,0} E (K_n)| = |E^1 (K) / \Gnorm_{n,0} E^1 (K_n)|$.
For $n \geq 0$, we denote by $j_n$ the map
$E^1 (K_n) \otimes_\mathbb{Z} \mathbb{Z}_p \to \mathcal{E} (K_n)$ 
induced from the embedding.

We remark that $j_0$ is not an isomorphism in general,
because $\mathcal{U}_0$ is not the group of semi-local units of $K$.
However, under our assumption that $\rankZ E(K) = \rankZ E (F) + 1$,
we can see that $j_0$ is an isomorphism.
Indeed, we see that
\[ \begin{array}{rcl}
\rankZp E^1 (K) \otimes_\mathbb{Z} \mathbb{Z}_p & = & \rankZ E(K) \\
 & = & \rankZ E (F) + 1 \\
 & = & \rankZp \mathcal{E} (F) + 1 \quad \text{(Leopoldt's conjecture holds for $F$)} \\
 & = & \rankZp \mathcal{U}_0 \\
 & = & \rankZp \mathcal{E} (K) \quad
\text{(it follows from the finiteness of $| \mathfrak{X}_\mathfrak{p} (K)|$),}
\end{array} \]
where $\rankZp$ denotes the $\mathbb{Z}_p$-rank  
(that is, the rank of the $\mathbb{Z}_p$-free part).
Let $\mu$ (resp. $\mu_*$) be the $\mathbb{Z}_p$-torsion submodule of 
$E^1 (K) \otimes_\mathbb{Z} \mathbb{Z}_p$ 
(resp. $\mathcal{E} (K)$).
We obtain the following commutative diagram with exact rows.
\[ \begin{CD}
0 @>>> \mu @>>> E^1 (K) \otimes_\mathbb{Z} \mathbb{Z}_p @>>>
(E^1 (K) \otimes_\mathbb{Z} \mathbb{Z}_p) /\mu @>>> 0 \\
@. @VV{j_0 |_\mu}V @VV{j_0}V @VVV @. \\
0 @>>> \mu_* @>>> \mathcal{E} (K) @>>>
\mathcal{E} (K) / \mu_* @>>> 0.
\end{CD} \]
We note that $j_0$ is surjective.
By the above result, we see that the right vertical map is 
isomorphism.
Moreover, the left vertical map is injective.
Hence we see that $j_0$ is an isomorphism.

Next, we shall consider the following diagram.
\[ \begin{CD}
E^1 (K_n) \otimes_\mathbb{Z} \mathbb{Z}_p @>{j_n}>> 
\mathcal{E} (K_n) \\
@VV{\text{norm}}V @VV{\text{norm}}V  \\
E^1 (K) \otimes_\mathbb{Z} \mathbb{Z}_p @>{j_0}>> 
\mathcal{E} (K).
\end{CD} \]
In our situation, this diagram is commutative.
(Note that the left vertical map is induced from the global norm, 
and the right vertical map is induced from the local norm.)
Since $j_n$ is a surjection and $j_0$ is an isomorphism, 
we see that both cokernels of the vertical maps are isomorphic.
From this, we can see that
\[ |E (K) / \Gnorm_{n,0} E (K_n)| = |E^1 (K) / \Gnorm_{n,0} E^1 (K_n)| =
| \mathcal{E} (K) / \Lnorm_{n,0} \mathcal{E} (K_n)|. \]
Then the claim follows.

By combining the equations \eqref{equation_local_units},
\eqref{equation_global_units} and Claim 2,  
we obtain the equation (\ref{eq_Bn_Dn}).

We shall show the second assertion. 
We first note that if $i_n (A (K)) \subset D (K_n)$ 
then $i_m (A (K)) \subset D (K_m)$ for all $m \geq n$.
Hence the equation (\ref{eq_Bn_Dn}) also holds for $m$ ($> n$).
Since 
\[ |\mathcal{U}_m / \mathcal{E} (K_m)| = |\Gal (M_\mathfrak{p} (K_m) /L(K_m))|, \quad
|D(K_m)| = |\Gal (L (K_m) /L' (K_m))|, \]
we see that 
\[ \begin{array}{rcl}
|\mathfrak{X}_\mathfrak{p} (K_m)| & = & 
|\Gal (M_\mathfrak{p} (K_m) /L(K_m))| \cdot |\Gal (L (K_m) /L' (K_m))| \cdot 
|\Gal (L' (K_m) /K_m)| \\
& = & |\mathfrak{X}_\mathfrak{p} (K)| \cdot |X' (K_m)| 
\end{array} \]
holds for all $m \geq n$.
\end{proof}

As a corollary to the above proposition, we also obtain the following.

\begin{cor}\label{Consequence_Fukuda_analog}
Let $n$ be a non-negative integer.
Suppose that $K$ satisfies the assumptions of Proposition \ref{Bn_Dn} at $n$.
Then the following are equivalent.
\begin{itemize}
\item[(i)] $|\mathfrak{X}_\mathfrak{p} (K_n)| = |\mathfrak{X}_\mathfrak{p} (K_{n+1})|$.
\item[(ii)] $|X' (K_n)| = |X' (K_{n+1})|$.
\item[(iii)] $M'_\mathfrak{p} (K_n) = L' (K_n)$.
\end{itemize}
In particular, if $A (K)$ is trivial and $|X|$ is finite,
then the stabilizations of $|\mathfrak{X}_\mathfrak{p} (K_m)|$
and $|X' (K_m)|$ occur at the same layer.
\end{cor}

\begin{proof}
The equivalence (i) $\Leftrightarrow$ (ii) immediately follows from the 
second assertion of Proposition \ref{Bn_Dn}.

The equivalence (ii) $\Leftrightarrow$ (iii) follows from
Theorem \ref{split_finiteness_criterion} (and a variant of Fukuda's theorem
\cite[Theorem 1(1)]{Fuku94}).
Indeed, assumptions of Proposition \ref{Bn_Dn} contain that
$i_n (A (K)) \subset D (K_n)$.
This implies that $i'_n (A' (K))$ is trivial.
Hence we can apply Theorem \ref{split_finiteness_criterion} for $n$.
The equivalence can be shown from this.
\end{proof}

Concerning Proposition \ref{Bn_Dn}, we give several remarks.
This proposition can be seen as 
one of relationships between our results and Fukuda-Taya's results. 
For the case where $K$ is a real quadratic field, 
Fukuda-Taya \cite{F-T} gave a result concerning the stabilization of $|X (K_n)|$.
Their result can be generalized as the following.

\begin{prop}[{cf. \cite[Lemma 11]{F-T}}]\label{criterion_Bn_Dn}
Assume that $|\mathfrak{X}_\mathfrak{p} (K)|$ is finite.
If $|D (K_n)| = |\mathfrak{X}_\mathfrak{p} (K)|$ for some $n$, then  
$|X (K_\infty)| = |X (K_n)|$. 
\end{prop}

\begin{proof}[Outline of the proof]
Assume that $|D (K_n)| = |\mathfrak{X}_\mathfrak{p} (K)|$.
Note that if this occurs, then 
$|A (K_n)^{\Gal (K_n/K)}| = |\mathfrak{X}_\mathfrak{p} (K)|$ 
is also satisfied since $D (K_n) \subset A (K_n)^{\Gal (K_n/K)}$ 
(recall also Lemma \ref{finiteness_B_n}).

The original argument is also applicable to our situation.
Take an arbitrary integer $m$ greater than $n$.
Then $A (K_n)^{\Gal (K_n/K)} = D (K_n)$ and 
$A (K_m)^{\Gal (K_m/K)} = D (K_m)$.
Since $|D (K_m)| = |D (K_n)|$, 
the norm map induces an isomorphism $D (K_m) \cong D (K_n)$.
This yields that the map 
$A (K_m)^{\Gal (K_m/K)} \to A (K_n)^{\Gal (K_n/K)}$ 
induced from the norm map is an isomorphism,  
and one can show $A (K_m) \cong A (K_n)$ from this
(cf. the proof of \cite[Theorem 2]{Gre76}).
\end{proof}

\begin{rem}
Suppose that $K$ and $n$ satisfy the assumptions of Proposition \ref{Bn_Dn}.
In this case, we can also show the assertion of Proposition \ref{criterion_Bn_Dn} 
by using our Theorem \ref{split_finiteness_criterion}.
Suppose that $|D(K_n)| = |\mathfrak{X}_\mathfrak{p} (K)|$.
Then, by Proposition \ref{Bn_Dn}, we see that
$M_\mathfrak{p} (K_n) = L (K_n)$.
Recall that $i_n (A (K)) \subset D(K_n)$ is satisfied  
(this is one of assumptions of Proposition \ref{Bn_Dn},
but this is actually a necessary condition for
$|D(K_n)| = |A (K_n)^{\Gal (K_n /K)}|$).
As mentioned in the proof of Corollary \ref{Consequence_Fukuda_analog},
this implies that $i'_n (A' (K))$ is trivial.
Then by Theorem \ref{split_finiteness_criterion},
$|X' (K_n)| = |X' (K_m)|$ is satisfied for all $m \geq n$.
Moreover, since $|D(K_n)| \leqq |D (K_m)|$,
we see that $|D(K_n)| = |\mathfrak{X}_\mathfrak{p} (K)|
= |D (K_m)|$.
From these equations, we conclude that
$|X (K_m)| = |X (K_n)|$ for all $m \geq n$.
\end{rem}

For the case of real quadratic fields, the converse of 
Proposition \ref{criterion_Bn_Dn} 
is essentially shown in \cite{FKKS}.
We can obtain a similar result in our situation.

\begin{prop}[{cf. \cite[Lemma 7.1 (1)]{FKKS}}]\label{converse_criterion_Bn_Dn}
Assume that $\rankZ E(K)=\rankZ E(F) + 1$.
If $|X(K_\infty)|= |X(K_n)|$ for some $n \geq 0$, then 
$|D (K_n) |= |\mathfrak{X}_\mathfrak{p} (K)|$ holds.
\end{prop}

\begin{proof}
Suppose that $|X(K_\infty)|= |X(K_n)|$. 
Then $i_m (A(K))$ is trivial for all sufficiently large $m$ (\cite[Proposition 2]{Gre76}).
We take $m$ such that $i_m (A(K))$ is trivial and $m > n$.
For $K_m$, we can see that $|D (K_m) |= |\mathfrak{X}_\mathfrak{p} (K)|$ 
by using Theorem \ref{finiteness_criterion1} and Proposition \ref{Bn_Dn}, 

We use the argument given in the proof of \cite[Lemma 7.1 (1)]{FKKS}.
Note that  $|X(K_m)| = |X(K_n)|$ holds.
Thus, the map $A(K_m) \to A(K_n)$ induced from the norm map is an isomorphism.
Hence the restriction of this map induces an isomorphism from $D(K_m)$ to $D(K_n)$.
The assertion follows from this.
\end{proof}

We also remark that for the case where $\rankZ E (K) > \rankZ E (F) +1$,
there is an example such that the equation \eqref{eq_Bn_Dn} of
Proposition \ref{Bn_Dn} does not hold
(even when all other assumptions are satisfied).

\begin{example}
Take a real quadratic field $k$ such that all of the following conditions
are satisfied.
\begin{itemize}
\item $p$ splits in $k$,
\item $A (k)$ is trivial,  
\item $A (k_1)$ is non-trivial and $D (k_1)$ is trivial,
\item $X (k_\infty)$ is finite.
\end{itemize}
We can find an example of such $k$ in, e.g., \cite{F-K} (see also Section \ref{section_quadratic}).
First, we note that Proposition \ref{Bn_Dn} is applicable for $k$.
That is,
\[ | \mathcal{U}_n /\mathcal{E} (k_n) | \cdot | D(k_n) | =
| \mathfrak{X}_\mathfrak{p} (k)| \]
holds for all $n$.
Moreover, since $D (k_1)$ is trivial, we obtain the inequality that
\[ | \mathfrak{X}_\mathfrak{p} (k)| = | \mathcal{U}_n /\mathcal{E} (k_n) |
< | \mathcal{U}_n /\mathcal{E} (k_n) | \cdot |A (k_1)|
= | \mathfrak{X}_\mathfrak{p} (k_1)|. \]
Next, we put $F = \mathbb{B}_1$ and $K = k_1$.
Then, $F$ and $K$ also satisfy the assumptions of Section \ref{assumptions}.
We note that $K_n = k_{n+1}$ and $K_\infty = k_\infty$.
We also note that $| \mathfrak{X}_\mathfrak{p} (K)|$ is finite.
Moreover, $i_n (A (K))$ is trivial for a sufficiently large $n$
because $| X (K_\infty) |$ is finite (\cite[Proposition 2]{Gre76}).
However, $\rankZ E (K) > \rankZ E (F) +1$.
By the above observation, we see that
$| \mathfrak{X}_\mathfrak{p} (k)| < | \mathfrak{X}_\mathfrak{p} (K)|$.
Hence,
\[ | \mathcal{U}_{n+1} /\mathcal{E} (K_n) | \cdot | D(K_n) | =
| \mathfrak{X}_\mathfrak{p} (k)| < | \mathfrak{X}_\mathfrak{p} (K)|. \]
\end{example}

It seems that our criterion is useful even if 
$A (K_n)^{\Gal (K_n / K)} \neq D(K_n)$ for all $n$.

\subsection{Fukuda-Taya's invariants and criteria}\label{subsection_FT}

In this subsection, we assume that 

\medskip

(R1) $k$ is a real quadratic field, and $p$ is an odd prime number which splits in $k$.

\medskip

We shall recall the invariants $n_0^{(r)}$, $n_2^{(r)}$
defined in Fukuda-Taya's paper \cite{F-T}.
These invariants play important roles for their criteria  
on Greenberg's conjecture.

\begin{definition}[cf. Fukuda-Taya \cite{F-T}, Taya \cite{Taya}]
Assume that $k$ and $p$ satisfy (R1).
Let $r$ be a non-negative integer.
We denote by $v_{\mathfrak{p}} ( \cdot )$ the normalized additive 
valuation corresponding to $\mathfrak{p}$.
Let $n_0^{(r)}$ be the minimal number of 
$v_{\mathfrak{p}} (\Gnorm_{r,0} (\beta)^{p-1} -1)$,
where $\beta$ runs runs $\mathfrak{p}'_r$-units of $k_r$,
Let $n_2^{(r)}$ be the minimal number of 
$v_{\mathfrak{p}} (\Gnorm_{r,0} (\varepsilon)^{p-1} -1)$,
where $\varepsilon$ runs units of $k_r$.
\end{definition}

Note that $n_0^{(r)} \geq n_2^{(r)}$.
The following is a partial result given in
Fukuda-Taya \cite{F-T} and Taya \cite{Taya}.

\begin{thm}[cf. Fukuda-Taya \cite{F-T}, Taya \cite{Taya}]\label{Fukuda_Taya}
Assume that $k$ and $p$ satisfy (R1).
Let $r$ be a positive integer.
Assume also that $A (k) = D (k)$, or $i_r (A (k))$ is trivial.
If $n_2^{(r)} = r +1$ or $n_0^{(r)} = r+1$, then
$| X (k_\infty) |$ is finite.
\end{thm}

\begin{rem}
If $A (k) = D (k)$, then $i_r (A (k))$ is trivial for all
sufficiently large $r$ (see \cite[pp.266--267]{Gre76}).
\end{rem}

For the case where $p=2$, a similar result to the above theorem is also known 
(Fukuda \cite{Fuku10}, Fukuda-Komatsu-Kumakawa-Sasaki \cite{FKKS}).

Concerning our results given in Section \ref{subsection_finiteness},
the following holds.

\begin{prop}\label{relation_Fukuda_Taya}
Assume that $k$ and $p$ satisfy (R1).
Let $r$ be a non-negative integer.

\smallskip

\noindent (1) If $M_\mathfrak{p} (k_r) = L(k_r)$, then $n_{2}^{(r)} = r+1$.

\smallskip

\noindent (2) If $M'_\mathfrak{p} (k_r) = L' (k_r)$, then $n_{0}^{(r)} = r+1$.
\end{prop}

\begin{proof}
We note that $M_\mathfrak{p} (k_r) = L(k_r)$ if and only if
$\mathcal{U}_r = \mathcal{E} (k_r)$.
Let $\Lnorm_{r,0}$ be the norm map from $(k_r)_{\mathfrak{p}_r}$ to
$k_\mathfrak{p}$.
We recall that $(\mathcal{U}_0 : \Lnorm_{r,0} \mathcal{U}_r) =p^r$.

We claim that if $\mathcal{U}_r = \mathcal{E} (k_r)$,  
then there is a unit $\varepsilon$ of $k_r$ such that 
$v_{\mathfrak{p}} (\Gnorm_{r,0} (\varepsilon)^{p-1} -1) = r+1$.
In the rest of this paragraph, we shall show this.
We may identify $\mathcal{U}_0$ with $1 + p \mathbb{Z}_p$.
We can take an element $u \in \mathcal{U}_r$ such that
$\Lnorm_{r,0} (u) = 1 + p^{r+1}$.
By the assumption,
for an arbitrary positive integer $n$,
there is a global unit $\varepsilon \in E (k_r)$
such that $\varepsilon^{p-1} u^{-1} \equiv 1 \pmod{\mathfrak{p}_r^n \mathcal{O}}$.
(Here, $\mathcal{O}$ is the valuation ring of $(k_r)_{\mathfrak{p}_r}$.)
If $n$ is sufficiently large, we see that  
\[ \Lnorm_{r,0} (1 + \mathfrak{p}_r^n \mathcal{O}) \subset 1+ p^{r+2} \mathbb{Z}_p. \]
Hence,
\[ \Lnorm_{r,0} (\varepsilon^{p-1} u^{-1}) \equiv 1 \pmod{p^{r+2} \mathbb{Z}_p}, \]
and this implies that
\[ \Gnorm_{r,0} (\varepsilon)^{p-1} \equiv 1 + p^{r+1} \pmod{p^{r+2} \mathbb{Z}_p}. \]
From this, we then see that
$v_{\mathfrak{p}} (\Gnorm_{r,0} (\varepsilon)^{p-1} -1) = r+1$.

We also note that $M_\mathfrak{p}' (k_r) = L' (k_r)$ if and only if
$\mathcal{U}_r = \mathcal{E}^\dagger (k_r)$, 
where $\mathcal{E}^\dagger (k_r) \subset \mathcal{U}_r$ is the 
similar object to $\mathcal{E} (k_r)$ defined by using 
$\mathfrak{p}'_r$-units of $k_r$ instead of global units (cf. \cite{Hachi}).
Hence, by using the same argument as above,
we can see that if $\mathcal{U}_r = \mathcal{E}^\dagger (k_r)$,
then $n_{0}^{(r)} = r+1$.
\end{proof}

It seems that there are other connections between the results given in 
\cite{F-T} and our results.
For example, the inequality $n_2^{(r)} \leq n_2^{(r+1)} \leq n_2^{(r)}+1$ 
given in \cite[p.278]{F-T} 
may relate to our Proposition \ref{inertia_subgroup_order}.
However, it would be long to explain them in detail, 
we omit to state here.

We would like to emphasize that our results hold not only for real quadratic fields.
In particular, our result covers the cases such that 
$|A (K_n)^{\Gal (K_n/K)}| = |D (K_n)|$ never holds.
(Showing this equality seems crucial in several places of \cite{F-T}.)

\subsection{Criteria for the stabilization of the $p$-rank of 
$A (K_n)$}\label{subsection_rank}
In this subsection, we shall show the criterion for the finiteness
of the $p$-rank of $X (K_\infty)$.

\begin{lem}\label{necessary_rank_stability}
If $\rankp \mathfrak{X}_\mathfrak{p} (K_n) > \rankp X (K_n)$,
then $\rankp X (K_\infty) > \rankp X (K_n)$.
\end{lem}

\begin{proof}
We can show the assertion by using the same type argument to
the proof of Lemma \ref{if_part_finite}.
That is,
\[ \rankp X (K_n) <
\rankp \mathfrak{X}_\mathfrak{p} (K_n) \leq
\rankp \mathfrak{X}_\mathfrak{p} (K_\infty)
= \rankp X (K_\infty). \]
\end{proof}

\begin{thm}\label{rank_criterion1}
Let $n$ be a positive integer.
Assume that $i_n (A (K))$ is trivial.
Then $\rankp \mathfrak{X}_\mathfrak{p} (K_n) = \rankp X (K_n)$
if and only if
$\rankp X (K_\infty) = \rankp X (K_n)$.
\end{thm}

\begin{proof}
The ``only if'' part follows from Lemma \ref{necessary_rank_stability}.

We shall show the ``if'' part.
To show this,
we shall mimic the method of the proof of Fukuda's theorem 
\cite[Theorem 1 (2)]{Fuku94}.
As noted in the proof of Theorem \ref{finiteness_criterion1}, we see that
$X (K_n) \cong X / \nu_n X$ under the assumption that $i_n (A (K))$ is trivial.
Hence, the condition 
$\rankp \mathfrak{X}_\mathfrak{p} (K_n) = \rankp X (K_n)$ is expressed as 
$\rankp X / \nu_n X = \rankp X / \omega_n X$. 
This implies that 
\[ \nu_n X + p X = \omega_n X + p X. \]
As same as Fukuda's proof, we put $Z = (\nu_n X + p X)/ p X$.
Then,
\[ T Z = T((\nu_n X + p X)/ p X) =
(\omega_n X + p X)/ p X = (\nu_n X + p X)/ p X = Z. \]
By using topological Nakayama's lemma, we see that $Z$ is trivial,
and hence $\nu_n X \subset p X$.
This yields that
\[ \rankp X (K_n) = \rankp X / \nu_n X \geq
\rankp X / p X = \rankp X. \]
From this, we conclude that $\rankp X (K_n) = \rankp X (K_\infty)$.
\end{proof}

We can also obtain the $X' (K_\infty)$-version of 
Theorem \ref{rank_criterion1}.
However, we omit to state the details.

\begin{rem}\label{Fukuda_analog_rank}
Similarly to that mentioned in Remark \ref{Fukuda_analog},
we can also obtain the $\mathfrak{X}_\mathfrak{p} (K_\infty)$-version
of Fukuda's theorem \cite[Theorem 1(2)]{Fuku94}.
That is, if $\rankp \mathfrak{X}_\mathfrak{p} (K_n)
= \rankp \mathfrak{X}_\mathfrak{p} (K_{n+1})$ then
$\rankp X = \rankp \mathfrak{X}_\mathfrak{p} (K_n)$.
Moreover, in this situation, we also see that
$\rankp X (K_{n+1}) =
\rankp \mathfrak{X}_\mathfrak{p} (K_{n+1})$.
We shall show this.
By using Lemma \ref{lem_cyclic_F_n} (2), we see that
\[ \rankp X (K_{m}) \leq
\rankp \mathfrak{X}_\mathfrak{p} (K_m) \leq \rankp X (K_{m}) +1 \]
for all $m$.
If $\rankp X (K_n) =  \rankp X (K_{n+1}) =
\rankp \mathfrak{X}_\mathfrak{p} (K_{n+1}) -1$, then
$\rankp X (K_m) < \rankp X$ for all $m \geq n$.
It is a contradiction, and then the assertion follows.
\end{rem}

Concerning the above remark, 
there is the case such that 
$\rankp \mathfrak{X}_\mathfrak{p} (K_n)
< \rankp \mathfrak{X}_\mathfrak{p} (K_{n+1})$ 
but $\rankp \mathfrak{X}_\mathfrak{p} (K_{n+1})
=\rankp X (K_{n+1})$.
See Example \ref{example_rank_7}.

We give one more remark.

\begin{rem}\label{rem_trivial_cyclicity}
Assume that $\mathfrak{X}'_\mathfrak{p} (K)$ is trivial.
Then by Corollary \ref{cor_Hachimori}(2), $\mathfrak{X}'_\mathfrak{p} (K_n)$ is
trivial for every $n$.
This implies that $\mathfrak{X}_\mathfrak{p} (K_n)$ is cyclic
because $\Gal (M_\mathfrak{p} (K_n) / M'_\mathfrak{p} (K_n))$ is
the decomposition subgroup of $\mathfrak{X}_\mathfrak{p} (K_n)$
for $\mathfrak{p}'_n$
\end{rem}

Hence, for the problem whether $\rankp X (K_\infty)$ is finite or not,  
it is sufficient to consider only the case that
$\mathfrak{X}'_\mathfrak{p} (K)$ is not trivial.

In the remaining part of this subsection, we shall give further criteria.
Some of them will be used in Section \ref{section_quadratic}.
We shall give results under the assumption that
``$|\mathfrak{X}_\mathfrak{p} (K)|$ is finite''.
If this is satisfied, 
we see can that $|\mathfrak{X}_\mathfrak{p} (K_n)|$ is also finite
for every positive integer $n$ 
by using Proposition \ref{inertia_subgroup_order}.

Before stating the criteria, we give an auxiliary result.
When $K$ is a real quadratic field, the behavior of 
$|A (K_n)^{\Gal (K_n/K)}|$ is well studied 
(see \cite{Fuku10}, \cite{F-K}, \cite{FKKS}, \cite{F-T}, \cite{Taya}).
The following is obtained inspired by the result for real quadratic fields.

\begin{lem}\label{unramified_subextension}
Assume that $\mathfrak{X}_\mathfrak{p} (K)$ is non-trivial and finite.
We put $|\mathfrak{X}_\mathfrak{p} (K)| = p^m$.
Assume also that $A (K)$ is trivial.
For an integer $n$ satisfying $0 < n \leq m$,
we denote by $L^{s}_n/K_n$ the maximal unramified subextension of
$M_\mathfrak{p} (K) K_n /K_n$.
Then, $[ L^{s}_n : K_n ] = p^n$.
\end{lem}

\begin{proof}
We put $G_n = \Gal (K_n /K)$.
In the proof of Lemma \ref{finiteness_B_n}, we showed that
$\mathfrak{X}_\mathfrak{p} (K_n)_{G_n} \cong 
\mathfrak{X}_\mathfrak{p} (K)$.
That is, $M_\mathfrak{p} (K) K_n$ is the maximal intermediate
field of $M_\mathfrak{p} (K_n) / K_n$ which is abelian over $K$.
Hence, $L^{s}_n$ is the intermediate field of $L (K_n)/K_n$
corresponding to $X (K_n)_{G_n}$.

We shall compute $|X (K_n)_{G_n}| = |A (K_n)^{G_n}|$ by using a well known formula.
In our situation, we see that
\[ |A (K_n)^{G_n}| = \dfrac{p^n}{(E (K) : E(K) \cap \Gnorm_{n,0} (K_n^\times))}, \]
where $\Gnorm_{n,0}$ denotes the norm map from $K_n^\times$ to $K^\times$.
To prove our assertion, it is sufficient to see that
$(E (K) : E(K) \cap \Gnorm_{n,0} (K_n^\times)) =1$.

To see the above equality, we shall show that every element of 
$E (K)$ is a local norm from $(K_n)_{\mathfrak{p}_n}$.
Since $K_n / K$ is unramified outside $\mathfrak{p}$, $\mathfrak{p}'$,
by using the product formula of the norm residue symbol
we see that every element of $E (K)$ is a local norm at every place.
Then, by Hasse's norm theorem,
we see that $E (K) \subset \Gnorm_{n,0} (K_n^\times)$.

Let $\Lnorm_{n,0}$ be the norm from
$(K_n)_{\mathfrak{p}_n} (= (F_n)_{\mathfrak{p}^0_n})$
to $K_{\mathfrak{p}} (= F_{\mathfrak{p}^0})$.
Since $(K_n)_{\mathfrak{p}_n} / K_{\mathfrak{p}}$ is
a totally ramified extension, we see that
$(\mathcal{U}_0 : \Lnorm_{n,0} \mathcal{U}_n) = p^n$.

Since $(E (K) : E^1 (K))$ is prime to $p$,
it is sufficient to show that every element of 
$E^1 (K)$ is a local norm.
Hence, if the inclusion relation
$\mathcal{E} (K) \subset \Lnorm_{n,0} \mathcal{U}_n$
is shown, this will be obtained.
Recall that
$\mathcal{E} (F) \subset \Lnorm_{n,0} \mathcal{E} (F_n)
\subset \Lnorm_{n,0} \mathcal{U}_n$
(see the proof of Proposition \ref{Bn_Dn}).
From the assumption that $A (K)$ is trivial,
we see that $( \mathcal{U}_0 : \mathcal{E} (K) )
= |\mathfrak{X}_\mathfrak{p} (K)| = p^m$.
By these facts, we can also obtain that
\[ ( \mathcal{U}_0 / \mathcal{E} (F) :
\Lnorm_{n,0} \mathcal{U}_n / \mathcal{E} (F)) = p^n, \quad
( \mathcal{U}_0 / \mathcal{E} (F) :
\mathcal{E} (K) / \mathcal{E} (F)) = p^m. \]
Since $\mathcal{U}_0 / \mathcal{E} (F) \cong \Gal (F_\infty /F)$
is a free $\mathbb{Z}_p$-module of rank $1$ and $n \leq m$,
wee see that $\mathcal{E} (K) / \mathcal{E} (F) \subset
\Lnorm_{n,0} \mathcal{U}_n / \mathcal{E} (F)$.
From this, we also can show that $\mathcal{E} (K) \subset
\Lnorm_{n,0} \mathcal{U}_n$.
Hence we obtained the desired result,
and this lemma follows.
\end{proof}

We can obtain the following criterion for the stabilization of $\rankp A (K_n)$.

\begin{prop}\label{rank_criterion_exponent}
Assume that $\mathfrak{X}_\mathfrak{p} (K)$ is a non-trivial
finite cyclic group.
Let $n$ be a positive integer.
We denote by $L^{s}_n/ K_n$ the maximal unramified subextension of
$M_\mathfrak{p} (K) K_n /K_n$.
Assume that both of the following conditions satisfied:
\begin{itemize}
\item $M_\mathfrak{p} (K) K_n \neq L^{s}_n$, and
\item the exponent of $A (K_n)$ is equal to $[L^{s}_n : K_n]$.
\end{itemize}
Then, $\rankp \mathfrak{X}_\mathfrak{p} (K_n) = \rankp X (K_n)$.
Moreover, if $i_n (A(K))$ is trivial, then
$\rankp X (K_\infty) = \rankp X (K_n)$.
\end{prop}

\begin{proof}
Assume for contradiction that
$\rankp \mathfrak{X}_\mathfrak{p} (K_n) > \rankp X (K_n)$.

First we note that $L^{s}_n /K_n$ is a non-trivial extension.
(Indeed, even if $A (K)$ is trivial, this fact follows from
Lemma \ref{unramified_subextension}.)
We put $p^m = [L^{s}_n : K_n]$.
Let $M_1 /K_n$ be the unique intermediate field of
$M_\mathfrak{p} (K) K_n / K_n$ whose degree over $L^{s}_n$ is $p$
(this exists by the assumption).
Then $[ M_1 : K_n ] =p^{m+1}$.
Since $\rankp \mathfrak{X}_\mathfrak{p} (K_n) > \rankp X (K_n)$,
there is a cyclic extension $M_2 / K_n$ of degree $p$
such that $\mathfrak{p}_n$ actually ramifies.
We put $M_3 = M_1 M_2$.
Note that $M_1 \cap M_2 = K_n$, and we see that
\[ \Gal (M_3 / K_n) \cong
\mathbb{Z} /p^{m+1} \mathbb{Z} \oplus \mathbb{Z} /p \mathbb{Z}, \quad
\Gal (M_3 / L^{s}_n) \cong
\mathbb{Z} /p \mathbb{Z} \oplus \mathbb{Z} /p \mathbb{Z}. \]
Take a generator $\sigma$ (resp. $\tau$) of $\Gal (M_3 / M_2)$
(resp. $\Gal (M_3 / M_1)$),
and put $\sigma_1 = \sigma^{p^m}$.
Then $\Gal (M_3 /K_1)$ is generated by $\sigma, \tau$,
and $\Gal (M_3 / L^{s}_n)$ is generated by $\sigma_1, \tau$.

By Lemma \ref{lem_cyclic_F_n} (2),
the inertia subgroup of $\mathfrak{X}_\mathfrak{p} (K_n)$
for $\mathfrak{p}_n$ is cyclic.
Since a prime lying above $\mathfrak{p}_n$ is totally ramified in 
both $M_1 / L^{s}_n$ and $M_2 L^{s}_n / L^{s}_n$, 
we see that $M_3 / L^{s}_n$ has an unramified subextension $L^{\square}_n/L^{s}_n$
of degree $p$.
Note that the fixed field of $M_3 / K_1$ by
$\langle \sigma_1 \rangle$ is $M_2 L^{s}_n$,
and $\mathfrak{p}$ ramifies in $M_2 L^{s}_n / L^{s}_n$.
Hence $L^{\square}_n$ is a fixed field of $\langle \sigma_1^a \tau \rangle$
with some $a$.
From this, we see that $L^{\square}_n /K_n$ is a cyclic extension of degree $p^{m+1}$.
Since $L^{\square}_n /K_n$ is unramified, this contradicts the assumption
on the exponent of $A (K_n)$.
Hence we obtained the equality
$\rankp \mathfrak{X}_\mathfrak{p} (K_n) = \rankp X (K_n)$.

The second assertion follows from Theorem \ref{rank_criterion1}.
\end{proof}

\begin{cor}\label{cor_rank_criterion_exponent}
Assume that $A (K)$ is trivial, and
$| \mathfrak{X}_\mathfrak{p} (K) | \geq p^{m + 1}$
with some positive integer $m$.
For an integer $1 \leq n \leq m$, if the exponent of $A (K_n)$
is less than or equal to $p^n$, then $\rankp X (K_\infty)
= \rankp X (K_n)$.
\end{cor}

\begin{proof}
We denote by $L^{s}_n/k_n$ the maximal unramified subextension of
$M_\mathfrak{p} (K) K_n /K_n$.
By Lemma \ref{unramified_subextension}, we see that $[L^{s}_n : K_n] =p^n$.
Then the assertion follows from Proposition \ref{rank_criterion_exponent}.
\end{proof}

The following three propositions will be 
used to prove Theorem \ref{theorem_cyclicity_real_quad}.

\begin{prop}\label{cyclicity_criterion1}
Assume that $|\mathfrak{X}_\mathfrak{p} (K)| \geq p^2$ and finite.
Assume also that $A (K)$ is trivial.
Then $\rankp X (K_\infty) =1$ if and only if $|A (K_1)| =p$.
\end{prop}

\begin{proof}

First, we assume that $|A (K_1)| = p$.
Then we see that 
$\rankp X (K_\infty) = \rankp X (K_1) = 1$ 
by Corollary \ref{cor_rank_criterion_exponent}.

Next, we assume that $|A (K_1)| \geq p^2$.
By the assumption, $M_\mathfrak{p} (K) K_1 / K_1$ is a
cyclic extension with degree greater than or equal to $p^2$.
Let $M_1 /K_1$ (resp. $L^{s}_1 /K_1$) be the unique intermediate field of
$M_\mathfrak{p} (K) K_1 / K_1$ with degree $p^2$ (resp. $p$).
Hence, by Lemma \ref{unramified_subextension},
$L^{s}_1 /K_1$ is an unramified extension, but $\mathfrak{p}_1$ ramifies in
$M_1 / L^{s}_1$.

We shall show that $\rankp \mathfrak{X}_\mathfrak{p} (K_1) > 1$.
Since this trivially holds if $A (K_1)$ is not cyclic,
we assume that $A (K_1)$ is cyclic in the following.
By the above facts, we see that $L (K_1) \cap M_1 = L^{s}_1$.
Since $L (K) \neq L^{s}_1$, we see that $\Gal (M_1 L(K_1) /K_1)$
is not cyclic, and then $\rankp \mathfrak{X}_\mathfrak{p} (K_1) > 1$.
Hence we see that $\rankp X (K_\infty) > 1$
because
$\mathfrak{X}_\mathfrak{p} (K_\infty) = X (K_\infty)$
and there is a surjection
$\mathfrak{X}_\mathfrak{p} (K_\infty) \to \mathfrak{X}_\mathfrak{p} (K_1)$.
\end{proof}

\begin{prop}\label{rank_criterion_p2}
Assume that $p=2$, $A (K)$ is trivial, and $|\mathfrak{X}_\mathfrak{p} (K)|=2$.
If $A (K_1)$ is a cyclic group satisfying $|A (K_1)| \geq 4$,
then $\rankp X (K_\infty) =1$.
\end{prop}

\begin{proof}
If $\mathfrak{p}'$ is not decomposed in
$M_\mathfrak{p} (K)/K$, then $X (K_\infty)$ is cyclic
(see Remark \ref{rem_trivial_cyclicity}).
Hence, it is sufficient to show the assertion under the condition
that $\mathfrak{p}'$ is decomposed in $M_\mathfrak{p} (K)/K$.

If $\rankp \mathfrak{X}_\mathfrak{p} (K_1) =1$, then the assertion
follows from Theorem \ref{rank_criterion1}.
We shall show this.

Before that, we remark that
$\rankp \mathfrak{X}_\mathfrak{p} (K_1)$ is at most $2$,
because $A (K_1) (\cong \Gal (L (K_1)/K_1))$ is cyclic,
and $\mathcal{U}_1 / \mathcal{E} (K_1) (\cong \Gal (M_\mathfrak{p} (K_1)/L(K_1)))$ is
cyclic.

Hence, in the following,
we assume that $\rankp \mathfrak{X}_\mathfrak{p} (K_\infty) = 2$.
Since $A (K_1)$ is cyclic, there is a cyclic extension
$M_2 / K_1$ of degree $p$ such that $\mathfrak{p}_1$ actually ramifies
(cf. the proof of Proposition \ref{rank_criterion_exponent}).
Since $M_\mathfrak{p} (K) K_1 / K_1$ is a quadratic extension, and 
it is unramified by Lemma \ref{unramified_subextension}.
We put $L^{s}_1 = M_\mathfrak{p} (K) K_1$.
We also denote by $L^{\diamond}_1 / K_1$ the unique
unramified cyclic extension of degree $4$.

Here, $M_2 L^{s}_1$ is the maximal elementary abelian $2$-extension
over $K_1$ unramified outside $\mathfrak{p}_1$, 
hence it is a Galois extension over $K$.
Hence $M_2 L^{s}_1 / M_\mathfrak{p} (K)$ is a quartic Galois extension,
and then it is an abelian extension.
(Actually, $\Gal (M_2 L^{s}_1 / M_\mathfrak{p} (K))$ is a Klein four group
because every prime lying above $\mathfrak{p}'$ ramifies in $L^{s}_1 / M_\mathfrak{p} (K)$,
but it is unramified in $M_2 L^{s}_1/ L^{s}_1$.)

On the other hand, since $L^{\diamond}_1 /K_1$ is the unique unramified
cyclic extension of degree $4$,
we see that $L^{\diamond}_1 /K$ is also a Galois extension.
From this, $L^{\diamond}_1 / M_\mathfrak{p} (K)$ is also
a quartic abelian extension.
(By using a similar argument given in the previous paragraph,
we can see that $\Gal (L^{\diamond}_1 / M_\mathfrak{p} (K))$ is
a Klein four group.)

We note that $M_2 L^{s}_1 / L^{s}_1$ is not unramified, and
$L^{\diamond}_1 /L^{s}_1$ is unramified.
Then $M_2 L^{s}_1 \cap L^{\diamond}_1 = L^{s}_1$, and hence
$M_2 L^{\diamond}_1 / M_\mathfrak{p} (K)$ is an abelian extension of
degree $8$.

As noted in the first paragraph of this proof,
it is sufficient to consider only the case that
there are two primes $\mathfrak{P}'_a$, $\mathfrak{P}'_b$
of $M_\mathfrak{p} (K)$ lying above $\mathfrak{p}'$.
The primes lying above $\mathfrak{p}'$ ramify in
$L^{s}_1 / M_\mathfrak{p} (K)$, and do not ramify in
$M_2 L^{\diamond}_1 / L^{s}_1$.
Hence, the inertia subgroup $I_a$ (resp. $I_b$)
of $M_2 L^{\diamond}_1 / M_\mathfrak{p} (K)$
for $\mathfrak{P}'_a$ (resp. $\mathfrak{P}'_b$) has order $2$.
Let $I$ be the subgroup generated by $I_a$ and $I_b$.
Then the order of $I$ is at most $4$.
The fixed field of $M_2 L^{\diamond}_1 / M_\mathfrak{p} (K)$ by $I$
is a non-trivial abelian $2$-extension which is unramified
outside the unique prime $\mathfrak{P}$ lying above $\mathfrak{p}$.

Here, since $M_\mathfrak{p} (K)/K$ is a cyclic extension,
we can see that there is no abelian $2$-extension
unramified outside $\mathfrak{P}$.
It is a contradiction, and hence
$\rankp \mathfrak{X}_\mathfrak{p} (K_1) = 1$.
\end{proof}

\begin{rem}
In the above proof of Proposition \ref{rank_criterion_p2},
we used the particularity of $p=2$ at the point that
there are only \textit{two} primes of $M_\mathfrak{p} (K)$
lying above $\mathfrak{p}'$.
\end{rem}

The following result can be seen as a counterpart to the above 
Proposition \ref{rank_criterion_p2} 
(however, we need an extra assumption).

\begin{prop}\label{rank_criterion_p2_2}
Assume that $p=2$, $A (K)$ is trivial, and 
$|\mathfrak{X}_\mathfrak{p} (K)|=2$.
Moreover, assume that $\rankZ E (K) = \rankZ E(F) + 1$.
If $|A (K_1)| = |A' (K_1)| = 2$, then $\rankp X (K_\infty) >1$.
\end{prop}

\begin{proof}
First of all, we show that 
$|\mathfrak{X}_\mathfrak{p} (K_1)| = 4$.
By Proposition \ref{inertia_subgroup_order},
we see that $|\mathfrak{X}_\mathfrak{p} (K_1)|$ is $2$ or $4$.
We note that $K$ satisfies the assumptions of Proposition \ref{Bn_Dn}.
By the assumption that $|A (K_1)| = |A' (K_1)| = 2$, 
we see that $D (K_1)$ is trivial.
Hence, by the equation (\ref{eq_Bn_Dn}) of Proposition \ref{Bn_Dn}, 
we see that $|\mathcal{U}_1 / \mathcal{E} (K_1)| = 2$.
This yields that $|\mathfrak{X}_\mathfrak{p} (K_1)| = 4$.

We assume that 
$X(=\mathfrak{X}_\mathfrak{p} (K_\infty)=X (K_\infty))$ is 
cyclic as a $\mathbb{Z}_2$-module, and we shall show that
this leads a contradiction.

Since $|\mathfrak{X}_\mathfrak{p} (K_1)| = 4$, 
the order of $X$ is at least $4$
(of course, $X$ may be infinite).
Let $a$ be a generator of $X$ as a $\mathbb{Z}_2$-module.
We recall that $\gamma$ is a fixed topological generator of $\Gal (K_\infty /K)$, 
and $\gamma$ acts on $X$. 
Hence there is an element $x$ of $\mathbb{Z}_2$ such that $\gamma (a) = x a$.
We denote by $\overline{\langle b \rangle}$ the $\mathbb{Z}_2$-submodule of $X$ 
generated by $b \in X$.

Recall that $X / \omega_0 X \cong \mathfrak{X}_\mathfrak{p} (K)$ 
(Theorem \ref{quotient_restricted} (1)).
Moreover, we see that $X / \nu_1 X \cong X (K_1)$.
(Since $A (K)$ is trivial, $Y = X$ holds in Theorem \ref{quotient_unramified} (1).)
Hence, we can see the following:
\begin{itemize}
\item $|X / (\gamma -1 ) X| = 2$ (since $|\mathfrak{X}_\mathfrak{p} (K)| = 2$),
\item $|X / (\gamma + 1 ) X| = 2$ (since $|X (K_1)| = 2$).
\end{itemize}
We note that $\omega_0$ corresponds to $\gamma -1$,
and $\nu_1$ corresponds to $\gamma + 1$ since $p=2$.

Since $|X / (\gamma -1 ) X| = 2$, we see that
$|\overline{\langle a \rangle} / \overline{\langle (x-1) a \rangle}| =2$.
This implies that $x$ must be written of the form $1 + 2u$ with
some unit $u$ of $\mathbb{Z}_2$.
Since $u$ is written of the form $1 +2 y$ with $y \in \mathbb{Z}_2$,
we see that $x = 3 + 4y$.

Moreover, since $|X / (\gamma + 1 ) X| = 2$, we see that
$|\overline{\langle a \rangle} / \overline{\langle (x+1) a \rangle}| =2$.
On the other hand, since
$x + 1 = 4 (1+y)$ and the order of $a$ is at least $4$,
we see that
$|\overline{\langle a \rangle} / \overline{\langle (x+1) a \rangle}| \geq 4$.
It is a contradiction.
Hence $X (K_\infty)$ is not cyclic.
\end{proof}

\section{Results for the case of $\mathbb{Q} (\sqrt{q})$ with $p=2$}\label{section_quadratic}

In this section, we fix $p = 2$.
Let $q$ be a prime number satisfying $q \equiv 1 \pmod{8}$.
We put $k= \mathbb{Q} (\sqrt{q})$, and 
we apply our results as $F = \mathbb{Q}$ and $K = k$.
Since $2(=p)$ splits in $k$, this situation satisfies the assumptions
stated in Section \ref{assumptions}.
In this case, we see that $A (k)$ is trivial.

It is known that the $\mu$-invariant of $k_\infty / k$ is zero 
(\cite[pp.10--11]{Iwa73mu}).
That is, $\ranktwo X (k_\infty)$ is finite.

It is also well known that $X (k_\infty)$ is trivial for the following cases
(see \cite{O-T}, \cite{Mo-Mo}):
\begin{itemize}
\item $q \equiv 9 \pmod{16}$ and $2^{\frac{q-1}{4}} \equiv  1 \pmod{q}$
\item $q \equiv 1 \pmod{16}$ and $2^{\frac{q-1}{4}} \equiv -1 \pmod{q}$
\end{itemize}
(We note that $2^{\frac{q-1}{4}} \equiv \pm 1 \pmod{q}$
since $q \equiv 1 \pmod{8}$.)
Moreover, it is also known that $X' (k_\infty)$ is trivial for the following case
(see, e.g., the proof of \cite[Theorem 4.1]{Mo-Mo}):
\begin{itemize}
\item $q \equiv 9 \pmod{16}$ and $2^{\frac{q-1}{4}} \equiv -1 \pmod{q}$
\end{itemize}
In this case, $X (k_\infty)$ is a finite cyclic group.

Hence, we assume that $q$ satisfies the following
condition throughout this section.

\medskip

(A1) $q \equiv 1 \pmod{16}$ and $2^{\frac{q-1}{4}} \equiv 1 \pmod{q}$.

\medskip

Under this condition, $A (k_1)$ is known to be non-trivial.
We can also see that $A' (k_1)$ is not trivial.
(It is known that $A (k_1)$ is a cyclic group and $|D(k_1)| \leq 2$.
Hence if $|A (k_1)| \geq 4$, then $A' (k_1)$ is not trivial.
For the case where $|A (k_1)| = 2$,
this follows from \cite[Theorem 5.5 (a)]{Y}.
See also \cite{Kuma25}.)

Moreover, the non-triviality of $A (k_1)$ implies the
non-triviality of $\mathfrak{X}_\mathfrak{p} (k)$  
(cf. \cite{Hachi}, or Corollary \ref{cor_Hachimori}).
We also remark that one can compute $|\mathfrak{X}_\mathfrak{p} (k)|$ 
by using the fundamental unit of $k$.
Actually, the invariant $n_2$ defined in \cite[p.22]{Fuku10}  
exactly satisfies $|\mathfrak{X}_\mathfrak{p} (k)|= 2^{n_2 - 2}$.

As well as Section \ref{section_criteria}, we often abbreviate 
$\mathfrak{X}_\mathfrak{p} (k_\infty) = X (k_\infty)$ as $X$.
Since $A (k)$ is trivial, we see that $X (k_n) \cong X / \nu_n X$. 
(Recall the proof of Proposition \ref{rank_criterion_p2_2}.)

\subsection{Yet another criterion for the cyclicity of $X (k_\infty)$}\label{subsection_cyclic}

As stated in the introduction, 
the necessary and sufficient condition for $X (k_\infty)$ being cyclic 
(as a $\mathbb{Z}_2$-module) is already studied by Mouhib-Movahhedi \cite{Mo-Mo}.
They gave a condition in terms of the the structure of 
$A (\mathbb{Q} (\sqrt{q (2 + \sqrt{2})}))$ (see \cite[Theorem 3.6]{Mo-Mo}).
Moreover, Mizusawa-Mouhib \cite[Theorem 4.2]{Mi-Mo} gave another 
condition in terms of a property of a primitive root modulo $q$.
Here we give yet another condition.

\begin{thm}\label{theorem_cyclicity_real_quad}
Assume that $q$ satisfies (A1).
Then the following holds.
\begin{itemize}
\item[(1a)] If $| \mathfrak{X}_\mathfrak{p} (k) | =2$ and
$|A (k_1)| =2$, then $\ranktwo X (k_\infty) >1$.
\item[(1b)] If $| \mathfrak{X}_\mathfrak{p} (k) | =2$ and
$|A (k_1)| >2$, then $\ranktwo X (k_\infty) =1$.
\item[(2a)] If $| \mathfrak{X}_\mathfrak{p} (k) | > 2$ and
$|A (k_1)| =2$, then $\ranktwo X (k_\infty) =1$.
\item[(2b)] If $| \mathfrak{X}_\mathfrak{p} (k) | > 2$ and
$|A (k_1)| >2$, then $\ranktwo X (k_\infty) >1$.
\end{itemize}
\end{thm}

\begin{proof}
The assertion (1a) (resp. (1b)) follows from 
Proposition \ref{rank_criterion_p2_2} (resp. \ref{rank_criterion_p2}).
Moreover, the assertions (2a) and (2b) follow from
Proposition \ref{cyclicity_criterion1},
\end{proof}

\begin{rem}\label{remark_2q}
We put $k' = \mathbb{Q} (\sqrt{2q})$.
It is known that $L(k_1) = L (k')$.
Hence whether $|A (k_1)| = 2$ or $> 2$ can be determined
by an object of the quadratic field $k'$
(see also \cite[Theorem 5.5 (b)]{Y}).
Our Theorem \ref{theorem_cyclicity_real_quad} implies that
we can determine whether $X (k_\infty)$ is cyclic as a
$\mathbb{Z}_2$-module or not by using objects of
two quadratic fields $k$, $k'$.
\end{rem}

We shall investigate the case where $\ranktwo X (k_\infty)=1$ a bit further.

\begin{example}\label{example_cyclic}
If $k$ satisfies the conditions of (1b) or (2a) of 
Theorem \ref{theorem_cyclicity_real_quad},
then $X (k_\infty)$ is cyclic.
We computed examples such that Theorem \ref{cyclic_finiteness_criterion}
is applicable at $k_1$.
This computation was done by using Magma \cite{Magma}.

\medskip

\noindent (1) Assume that $k$ satisfies the conditions of (1b).
In this case, $|Z_\mathfrak{p} (k_n)| \leq 2$ holds for all $n$.
Hence if $M_\mathfrak{p} (k_n) \neq M'_\mathfrak{p} (k_n)$ is satisfied
for some $n$, then $M'_\mathfrak{p} (k_n) = L' (k_n)$ is also satisfied  
(this implies that $|X' (k_\infty)| = |X' (k_n)|$).
Note that $M_\mathfrak{p} (k) = M'_\mathfrak{p} (k)$
always holds for this case.
In the range $q < 300,000$, there are $132$ fields which satisfies the 
condition of (1b), and $69$ of them satisfy
$M_\mathfrak{p} (k_1) \neq M'_\mathfrak{p} (k_1)$.

\medskip

\noindent (2) Assume that $k$ satisfies the conditions of (2a).
In our computation, we previously excluded the fields satisfying either 
of the following conditions.
\begin{itemize}
\item $M_\mathfrak{p} (k) \neq M'_\mathfrak{p} (k)$. 
\item $D (k_1)$ is non-trivial (this can be checked by using $k' = \mathbb{Q} (\sqrt{2q})$).
\end{itemize}
For these cases, Theorem \ref{cyclic_finiteness_criterion} is applicable 
without computing $M_\mathfrak{p} (k_1)$.
In the range $q < 300,000$, the number of fields which are not 
excluded is $144$, 
and $73$ of them satisfy
$M_\mathfrak{p} (k_1) \neq M'_\mathfrak{p} (k_1)$ 
(hence $|X' (k_\infty)| = |\mathfrak{X}'_\mathfrak{p} (k_1)|$).
\end{example}

We recall that Greenberg's conjecture was already confirmed 
in the range $q < 1,000,000$ (\cite{FKKS}).

\begin{rem}
Suppose that $K/F$ and $p$ satisfy the assumptions of Section \ref{assumptions} 
($K$ is not restricted to a real quadratic field).
Assume for simplicity that $A (K)$ is trivial.
For the finiteness of $|X' (K_\infty)|$, we have two criteria.
That is, if either
\[ M'_\mathfrak{p} (K_n) = L' (K_n) \quad \text{or} \quad
|\mathfrak{X}'_\mathfrak{p} (K_n)|
= |\mathfrak{X}'_\mathfrak{p} (K_{n+1})| \]
is satisfied, then $|X' (K_\infty)|$ is finite 
(Theorem \ref{split_finiteness_criterion} and Remark \ref{Fukuda_analog}).
Here is a question: does one condition include the other? 
The answer is No.
Indeed, in the above Example \ref{example_cyclic} (1), 
there are examples such that 
\[ |\mathfrak{X}'_\mathfrak{p} (k)|
< |\mathfrak{X}'_\mathfrak{p} (k_1)|, \quad \text{but} \quad 
M'_\mathfrak{p} (k_1) = L' (k_1) \]
(e.g., $q= 2593, 5297, 17737$).
On the other hand, in Example \ref{example_cyclic} (2), 
there are examples such that 
\[ |\mathfrak{X}'_\mathfrak{p} (k)|
= |\mathfrak{X}'_\mathfrak{p} (k_1)|, \quad \text{but} \quad 
M'_\mathfrak{p} (k_1) \neq L' (k_1) \]
(e.g., $q = 7393, 13441, 14449$).
\end{rem}

In the rest of this subsection, 
we shall give several remarks concerning the $2$-rank of $X (k_\infty)$.

\begin{rem}\label{rank_upper_bound}
By considering $|A (k_n)^{\Gal (k_n /\mathbb{B}_n)}|$, 
we can see that $\ranktwo A(k_n) \leq 2^n -1$ holds in general 
(see, e.g., the proof of \cite[Lemma 2.1]{Kuma25}).
Hence, we also see that 
$\ranktwo \mathfrak{X}_\mathfrak{p} (k_n) \leq 2^n$.
Moreover, we see that $\ranktwo X (k_\infty) \leq 2^e-1$, 
where $e = v_2 (q-1) -2$ and $v_2 ( \cdot )$ is the normalized additive 
$2$-adic valuation.
\end{rem}

We note that the condition such that $\ranktwo X (k_\infty) =2$ is 
already studied by Mizusawa-Mouhib \cite{Mi-Mo}.
They obtained the following result.
We will use this later.

\begin{prop}[{see the proof of \cite[Theorem 4.2]{Mi-Mo}}]\label{rank_X_2}
Assume that $q$ satisfies (A1).
Then $\ranktwo X (k_\infty) =2$ if and only if $\ranktwo A (k_2) = 2$.
\end{prop}

It is also mentioned in the proof of \cite[Theorem 4.2]{Mi-Mo} 
that $\ranktwo X (k_\infty) \geq 3$ if and only if $\ranktwo A (k_2) = 3$. 
Recall that $\ranktwo A (k_2) \leq 3$, and hence 
$\ranktwo \mathfrak{X}_\mathfrak{p} (k_2) \leq 4$.
We give an additional result.

\begin{cor}\label{cor_rank_2}
Assume that $q$ satisfies (A1).
Then $\ranktwo X (k_\infty) \geq 4$ if and only if 
$\ranktwo \mathfrak{X}_\mathfrak{p} (k_2) = 4$.
\end{cor}

\begin{proof}
Assume that $\ranktwo \mathfrak{X}_\mathfrak{p} (k_2) \leq 3$.
We see that if $\ranktwo A (k_2) \leq 2$ then $\ranktwo X (k_\infty) \leq 2$ 
by the fact stated the above.
Moreover, if $\ranktwo A (k_2) = 3$ then 
$\ranktwo X(k_\infty) = 3$ by Theorem \ref{rank_criterion1}.

Conversely, if $\ranktwo \mathfrak{X}_\mathfrak{p} (k_2) = 4$, 
then $\ranktwo X(k_\infty) \geq 4$ by Lemma \ref{necessary_rank_stability}.
\end{proof}

\subsection{Sufficient conditions for the finiteness 
of $X (k_\infty)$}\label{subsection_GC}

In this subsection, we mainly treat the case where $k$ satisfies 
the following: 

\medskip

(A2) $|\mathfrak{X}_\mathfrak{p} (k)|=2$ and $|A (k_1)|=2$.

\medskip

\noindent 
Under the assumptions (A1) and (A2), 
$A (k_n)$ is not cyclic for all $n \geq 2$ by Theorem 
\ref{theorem_cyclicity_real_quad} and Fukuda's theorem \cite[Theorem 1 (2)]{Fuku94}.
We also note that $|A' (k_1)| =2$ and $|\mathfrak{X}_\mathfrak{p} (k_1)| = 4$.

\begin{lem}\label{finiteness_Dn_nontrivial}
Assume that $q$ and $k$ satisfy (A1) and (A2).
Let $n$ be a positive integer.
Then $D (k_n)$ is not trivial if and only if
$X (k_\infty) \cong X (k_n)$.
\end{lem}

\begin{proof}
Since $|\mathfrak{X}_\mathfrak{p} (k)|=2$, the assertion 
follows from Propositions \ref{criterion_Bn_Dn} and 
\ref{converse_criterion_Bn_Dn}.
\end{proof}

It seems difficult to show the finiteness of $|X (k_\infty)|$ 
by observing $k_1$ only.
In the following, we shall focus on $k_2$.
Under the conditions (A1) and (A2), 
it can be shown that $|A (k_2)| \geq 8$.
(We shall give a slightly stronger result.)

\begin{prop}\label{k2_non_cyclic}
Assume that $q$ and $k$ satisfy (A1) and $|A(k_1)|=2$.
Then $A (k_2)$ is not cyclic if and only if $|A (k_2)| \geq 8$.
\end{prop}

\begin{proof}
First, we shall show the ``only if'' part.
Assume that $A (k_2)$ is not cyclic.
Then $|A (k_2)| \geq 4$ holds.
Now we suppose that $|A (k_2)| = 4$.
Put $k'_1 = \mathbb{Q} (\sqrt{q (2+ \sqrt{2})})$,  
which is the unique intermediate field
of $k_2 / \mathbb{B}_1$ different from $k_1$, $\mathbb{B}_2$.
Since $X (k_\infty)$ is not cyclic, we can see that $|A (k'_1)| \geq 8$ 
by using \cite[Proposition 3.6]{Mo-Mo}.
Recall that $k_2 / k'_1$ is an unramified quadratic extension,
and we then see that $|A (k'_1)| = 8$ and $L (k_2) = L (k'_1)$.

We denote by $\gamma^*$ (resp. $\sigma$) the generator of
$\Gal (k_2/k_1)$ (resp. $\Gal (k_2 / \mathbb{B}_2)$).
We note that $(\sigma + 1) A (k_2)$ is trivial 
because it is the image of the map 
$A (\mathbb{B}_2) \to A (k_2)$ induced from the extension of ideals.
Using this fact, we see that 
\[ (\gamma^* \sigma -1) A (k_2) = 
(\gamma^* \sigma + \gamma^* - \gamma^* -1) A (k_2)
= (- \gamma^* -1) A (k_2). \]
The fact $L (k_2) = L (k'_1)$ implies that 
$(\gamma^* \sigma -1) A (k_2)$ is trivial, and 
then $(\gamma^* +1) A (k_2)$ is trivial.
Similarly to the above, $(\gamma^* +1) A (k_2)$ is the image of 
the map $i_{1,2} : A (k_1) \to A (k_2)$ induced from the extension of ideals.
Hence, $i_{1,2}$ is the zero map.
From this, we can see that $X (k_\infty) \cong X (k_2)$ 
by using \cite[Proposition 4.5]{B-C}.
Furthermore, by Lemma \ref{finiteness_Dn_nontrivial},
we see that $D (k_2)$ is not trivial.

On the other hand, we can show that $D (k_2)$ is trivial
by using the following argument.
In $k'_1$, there is a unique prime $\mathscr{P}$ 
lying above $2$, and 
there are two primes $\mathscr{Q}_1$, $\mathscr{Q}_2$ lying above $q$.
We denote by $c (\mathscr{P})$, $c (\mathscr{Q}_1)$, $c (\mathscr{Q}_2)$ 
the ideal classes of $k_2$ including $\mathscr{P}$, $\mathscr{Q}_1$, $\mathscr{Q}_2$, 
respectively.
We can show that these classes have order at most $2$ 
(because the class number of $\mathbb{B}_1$ is $1$ and 
the primes lying above $2$ or $q$ ramify in $k'_1 / \mathbb{B}_1$).
It can be shown that
\[ \sqrt{q (2+ \sqrt{2})} O_{k'_1} =
\mathscr{P} \mathscr{Q}_1 \mathscr{Q}_2. \]
Hence the product $c (\mathscr{P}) c (\mathscr{Q}_1) c (\mathscr{Q}_2)$ is trivial.

We can also show that every prime lying above $q$ splits in $L(k_1) / k_1$.
(See the proof \cite[Lemma 4.2]{Kuma25}.
Note that the same argument works for our situation.)
Hence every prime of $k_2$ lying above $q$ splits in
the quadratic extension $L(k_1) k_2 /k_2$.
Since both $\mathscr{Q}_1$, $\mathscr{Q}_2$ split in $k_2 / k'_1$,
they also split completely in $L(k_1) k_2 /k'_1$.

We note that $k'_1$ and $L(k_1) k_2$ are Galois extensions over $\mathbb{Q}$.
From the above result,
the possibility of the inertia subgroups of $\Gal (L(k'_1) / k'_1)$
for $\mathscr{Q}_1$, $\mathscr{Q}_2$ are the following:
\begin{itemize}
\item both are trivial.
\item both are $\Gal (L(k'_1) / L(k_1) k_2)$.
\end{itemize}
In either case, we can conclude that $c (\mathscr{Q}_1) c (\mathscr{Q}_2)$ is trivial.
This implies that $c (\mathscr{P})$ trivial. and hence the primes
of $k_2$ lying above $2$ splits completely in $L(k'_1) (=L(k_2))$.
That is, $D (k_2)$ is trivial, and it is a contradiction.
Hence $| A (k_2) | \geq 8$.

Next, we shall show the ``if'' part.
Assume that $| A (k_2) | \geq 8$.
We also use the symbols of the previous paragraphs.
By using the same argument to the above, we see that  
\[ |(\gamma^* \sigma -1) A (k_2)| = |(\gamma^* +1 ) A (k_2)| \leq 2. \]
Note that the intermediate field of $L(k_2) /k_2$ corresponding to 
$(\gamma^* \sigma -1) A (k_2)$ is the maximal unramified subextension of $k_2$ 
which is abelian over $k'_1$ 
(hence we showed the degree of this subextension is greater than or equal to $4$).
Since $k_2 / k'_1$ is unramified, we see that $|A (k'_1)| \geq 8$.
Then, by using \cite[Proposition 3.6]{Mo-Mo}, we conclude that 
$A (k_2)$ is not cyclic.
\end{proof}

By combining Theorem \ref{theorem_cyclicity_real_quad} and the second author's 
result given in \cite{Kuma25}, 
we obtain the following sufficient condition on the finiteness
of $X (k_\infty)$.

\begin{cor}\label{cor_full_rank}
Assume that $q$ and $k$ satisfy (A1) and $| A (k_1) | =2$.
We put $e = v_2 (q-1) -2$, where $v_2 ( \cdot )$ is the normalized additive 
$2$-adic valuation.
If $\ranktwo A(k_e) = 2^{e} - 1$, then $|X (k_\infty)| = |X (k_e)|$.
\end{cor}

\begin{proof}
By Theorem \ref{theorem_cyclicity_real_quad}, 
we see that $|\mathfrak{X}_\mathfrak{p} (k)|$ must be $2$ 
because $| A (k_1) | =2$ and \\
$\ranktwo X (k_\infty) > 1$.
Hence $k$ satisfies (A2).

It is already known that if $\ranktwo A(k_e) = 2^{e} - 1$
then $D (k_e)$ is not trivial (\cite[Theorem 1.5]{Kuma25}).
Thus the assertion follows from Lemma \ref{finiteness_Dn_nontrivial}.
\end{proof}

Recall that $\ranktwo A(k_n) \leq 2^n -1$ holds.
We shall introduce the following criterion to 
decide whether $\ranktwo A(k_n)$ is maximal or not.
(Actually, we can determine $\ranktwo A(k_n)$ by computing the 
ray class group of $\mathbb{B}_n$ modulo $q O_{\mathbb{B}_n}$.
See \cite[Corollary 2.2]{Mi-Mo}.
The following is a variant of this result.)

\begin{lem}
Assume that $q$ and $k$ satisfy (A1) and (A2).
Let $n$ be a positive integer.
Assume also that $q$ splits completely in $\mathbb{B}_n$.
Let $\mathfrak{Q}$ be a prime of $\mathbb{B}_n$ lying above $q$ 
(any one of the $2^n$ primes), and 
$L_\mathfrak{Q} (\mathbb{B}_n)$ the maximal abelian $2$-extension 
unramified outside $\mathfrak{Q}$.
Then $\ranktwo A(k_n)=2^n -1$ if and only if 
$L_\mathfrak{Q} (\mathbb{B}_n) / \mathbb{B}_n$ is non-trivial.
\end{lem}

\begin{proof}[Outline of the proof]
The ``only if'' part follows from the argument given in \cite[p.51]{Kuma25}.
(Strictly speaking, in \cite{Kuma25}, it is assumed that 
the decomposition field for $q$ of $\mathbb{B}_\infty / \mathbb{B}$ 
must be $\mathbb{B}_n$.
However, we can show the assertion under our assumption.)

We shall show the ``if'' part.
Let $L_q / \mathbb{B}_n$ be the the maximal elementary abelian 
$2$-extension unramified outside $q$.
By the non-triviality of $L_\mathfrak{Q} (\mathbb{B}_n) / \mathbb{B}_n$, 
we can show that the degree $[L_q : \mathbb{B}_n]$ is $2^{2^n}$.
$k_n$ is an intermediate field of  $L_q / \mathbb{B}_n$.
We note that $L_q / k_n$ is an unramified extension, 
because all primes lying above $q$ ramify in  $L_q / \mathbb{B}_n$ 
and the inertia subgroup of $\Gal(L_q / \mathbb{B}_n)$ for each prime is cyclic.
Hence $L_q / k_n$ is an unramified elementary abelian $2$-extension of 
degree $2^{2^n -1}$, and this yields that $\ranktwo A(k_n)=2^n -1$.
\end{proof}

\begin{example}\label{example_rank_7}
For the case where $q$ satisfies (A1), (A2), and $q \equiv 33 \pmod{64}$, 
if $\ranktwo A(k_3) = 7$ then $|X (k_\infty)| = |X (k_3)|$.
We computed the ray class group of $\mathbb{B}_3$ 
(modulo a prime lying above $q$) by using PARI/GP \cite{PARI}, 
and checked that $70$ prime numbers $q$ satisfy $\ranktwo A(k_3) = 7$ in the 
range $q < 10,000,000$ 
($q= 118369, 128033, 143137, \ldots, 9800801$).
We note that $\ranktwo \mathfrak{X}_\mathfrak{p} (k_2) \leq 4$ always holds.
Hence these examples satisfy 
\[ \ranktwo \mathfrak{X}_\mathfrak{p} (k_2) < 
\ranktwo \mathfrak{X}_\mathfrak{p} (k_3) \quad \text{and}  
\quad \ranktwo \mathfrak{X}_\mathfrak{p} (k_3) = \ranktwo A (k_3). \]
(These examples also satisfy 
both $|\mathfrak{X}_\mathfrak{p} (k_2)| < |\mathfrak{X}_\mathfrak{p} (k_3)|$ and 
$|\mathfrak{X}'_\mathfrak{p} (k_2)| < |\mathfrak{X}'_\mathfrak{p} (k_3)|$, 
but $|\mathfrak{X}_\mathfrak{p} (k_3)| = |X (k_3)|$ holds.)
\end{example}

The following seems useful for observing the $\lambda$-invariant in our situation.

\begin{prop}\label{prop_A2}
Assume that $p$, $k$ satisfy (A1) and (A2).
We denote by $\lambda$ the $\lambda$-invariant of $k_\infty  /k$.

\smallskip

\noindent (1) Either of the following holds: 
“$\mathfrak{X}_\mathfrak{p} (k_\infty)$ is finite” or 
“$\mathfrak{X}_\mathfrak{p} (k_\infty)$ has no nontrivial 
finite $\Lambda$-submodule”.

\smallskip

\noindent (2) For every positive integer $n$, $\lambda$ never satisfies 
$0 < \lambda < \ranktwo \mathfrak{X}_\mathfrak{p} (k_n)$.
In particular, $\lambda$ is never equal to $1$.
\end{prop}

\begin{proof}
Since $|X / \omega_0 X| = 2$, 
the assertion of (1) can be shown by
using the same argument as given in \cite{Oza97cyc}
(see the proof of Theorems 1 and 2 of \cite{Oza97cyc}).

We shall show (2).
Assume that $0 < \lambda < \ranktwo \mathfrak{X}_\mathfrak{p} (k_n)$.
Then, by (1) and the fact that $X$ is finitely generated over $\mathbb{Z}_2$, 
we see that $X$ is a free $\mathbb{Z}_2$-module of rank $\lambda$.
However, since
$\ranktwo \mathfrak{X}_\mathfrak{p} (k_n) \leq \ranktwo X$, 
it is a contradiction.
The first assertion follows.

We recall that $\ranktwo X \geq 2$.
Then the second assertion of (2) also follows from the above argument.
\end{proof}

Hence, under the assumptions (A1) and (A2), 
if the inequality $\lambda < \ranktwo \mathfrak{X}_\mathfrak{p} (k_n)$ is shown 
for some $n$, then $\lambda =0$ is derived.

\begin{thm}\label{theorem_second_layer}
Assume that $q$ and $k$ satisfy (A1).
Assume also that $|A (k_1)| =2$, and  
\[ A (k_2) \cong \mathbb{Z} / 2 \mathbb{Z} \oplus \mathbb{Z} / 4 \mathbb{Z} \]
as an abelian group.
Then, $|X (k_\infty)|$ is finite.
\end{thm}

\begin{proof}
As same as the proof of Corollary \ref{cor_full_rank}, we see that 
$| \mathfrak{X}_\mathfrak{p} (k) | =2$ because $A (k_2)$ is not cyclic.

We denote by $\lambda$ the $\lambda$-invariant of $k_\infty  /k$.
By Proposition \ref{rank_X_2}, we see that $\lambda \leq 2$
(recall that $\lambda \leq \ranktwo X (k_\infty)$ holds).
Moreover, the case $\lambda = 1$ does not occur by Proposition \ref{prop_A2} (2).
Hence it is sufficient to show that the case $\lambda = 2$ does not occur.

We put $\Gamma_n = \Gal (k_\infty / k_n)$.
Let $X^{\Gamma_n}$ be the $\Gamma_n$-invariant submodule of $X$.
We also write $X / \omega_n X$ as $X_{\Gamma_n}$ in this proof.

We denote by $X_{\mathrm{fin}}$ the maximal finite $\Lambda$-submodule of $X$.
It can be shown that 
$X^{\Gamma_n} = (X_{\mathrm{fin}})^{\Gamma_n}$ for all $n$ 
(see, e.g., \cite[p.225]{Oza97cyc}).
Hence if $\lambda > 0$, then $X^{\Gamma_n}$ is trivial 
by Proposition \ref{prop_A2}.

Assume that $\lambda =2$.
Let $F(T)$ be the characteristic polynomial of
$X$ as a $\Lambda$-module
(in the sense of \cite[(5.3.9) Definition]{NSW}).
We may assume that $F(T)$ is a monic distinguished polynomial
with degree $2$.
We can write
\[ F(T) = T^2 + a T + b \]
with $a, b \in 2 \mathbb{Z}_2$.
It is known that
\[ \frac{|X^{\Gamma_n}|}{|X_{\Gamma_n}|}
= \prod_{\zeta} | F(\zeta -1) |_2, \]
where $\zeta$ runs all elements satisfying $\zeta^{2^n} =1$ and
$| \cdot |_2$ denotes the normalized multiplicative $2$-adic absolute value.
(See \cite[Chapter IV, \S 3, Exercise 3]{NSW}.
Here we use the multiplicative absolute value to match the notation of this exercise.)
We recall that $|X^{\Gamma_n}|=1$.

Since $|X_{\Gamma_0}| = |\mathfrak{X}_\mathfrak{p} (k)| =2$,
we see that
\[ | F(0) |_2 = \dfrac{1}{2}. \]
Hence $|b|_2 = \dfrac{1}{2}$.
Moreover, since $|X_{\Gamma_1}| = |\mathfrak{X}_\mathfrak{p} (k_1)| =4$,
we see that
\[ | F(0) |_2 \cdot | F(-2) |_2 = \dfrac{1}{4}. \]
Let $i$ be a primitive fourth root of unity.
Then we see that
\[ F(i-1) F(-i-1) = (b-a)^2 + (a-2)^2. \]
Under the condition that $|b|_2 = \dfrac{1}{2}$, we can show that
$|(b-a)^2 + (a-2)^2|_2 \leq \dfrac{1}{8}$.
Hence, by summarizing these, we obtain that
\[ | F(0) |_2 \cdot | F(-2) |_2 \cdot | F(i-1) |_2 \cdot | F(-i-1) |_2
\leq \dfrac{1}{32}. \]
This implies that
\[ | X_{\Gamma_2} | = |\mathfrak{X}_\mathfrak{p} (k_2)| \geq 32. \]
We also note that $|Z_\mathfrak{p} (k_2)| \leq 2$, then we see that  
$|X (k_2)| \geq 16$.
This is a contradiction.
\end{proof}

By combining the above results, we obtain the following:

\begin{cor}[a weak version of {\cite[Theorem 1.1]{Sasaki}}]
Assume that $q$ and $k$ satisfy (A1).
If $q \equiv 17 \pmod{32}$, $|A (k_1)| = 2$, and $|A (k_2)| = 8$,  
then $|X (k_\infty)|$ is finite.
\end{cor}

\begin{proof}
By Proposition \ref{k2_non_cyclic}, we see 
that $\ranktwo A(k_2)$ must be $2$ or $3$.
When $\ranktwo A(k_2) = 2$, the assertion follows from 
Theorem \ref{theorem_second_layer}.
When $\ranktwo A(k_2) = 3$, the assertion follows from 
Corollary \ref{cor_full_rank}.
\end{proof}

\begin{rem}\label{rem_previous_results}
For the above corollary, a slightly stronger result is already shown.
That is, under the assumption of this corollary, actually 
$|X(k_\infty)| =8$ holds.
(Our corollary asserts only the finiteness of $|X(k_\infty)|$ 
for the case where $\ranktwo A(k_2) = 2$.)
This fact was first proven by the second author under a certain assumption.
It was obtained around 2025, but he had not written down as a paper.
Recently, Sasaki \cite{Sasaki} obtained the results including this fact 
(without extra assumptions).
For Sasaki's result on this fact, 
see \cite[Example 6.8 and Tables 3,4]{Sasaki}.
Note also that both Sasaki's proof and the second author's original proof 
use cyclotomic units mainly.
\end{rem}

In closing the present paper, we propose some conjectures.
We computed several data of the second layer $k_2$ 
for the case where $q$ and $k$ satisfy (A1) and (A2) 
in the lange $q < 300,000$.
(Here we used PARI/GP \cite{PARI} and the ideal class groups are
computed \textbf{under GRH}.
Recall also that the validity of Greenberg's conjecture is already confirmed 
in this range (\cite{FKKS}).)
As a result, we found several remarkable phenomena in the data.
We expect that these holds in general, and state them as the following:

\begin{conjintro}
Assume that $q$ and $k$ satisfy (A1) and (A2).
Then the following hold.
\begin{itemize}
\item[(i)] $|A (k_2)|$ is $8$ or $16$.
\item[(ii)] If $\ranktwo A (k_2) =3$ then $|A (k_2)|=8$.
\item[(iii)] Assume furthermore that $q \equiv 1 \pmod{32}$, 
then $|X (k_\infty)|=|X (k_2)|$ if and only if $\ranktwo A (k_2) =2$.
\end{itemize}
\end{conjintro}

We mention that the statement (i) had already been conjectured by the second author 
before conducting the above computation.

\begin{acknowledgement}
The authors would like to express their thanks to Sosuke Sasaki 
for answering our inquiry on the preprint \cite{Sasaki}.
The authors also would like to express their thanks to Yasushi Mizusawa 
for giving comments on the earlier version of the manuscript.
The first author is grateful to the authors of \cite{FKKS} for 
generously sharing their preprint. 
\end{acknowledgement}

\bigskip

\begin{flushleft}
Tsuyoshi Itoh \\
Division of Mathematics, 
Education Center,
Faculty of Innovative Management Science, \\
Chiba Institute of Technology, \\
2--1--1 Shibazono, Narashino, Chiba, 275--0023, Japan \\
e-mail : \texttt{tsuyoshi.itoh@it-chiba.ac.jp}

\bigskip

Naoki Kumakawa \\
Department of Mathematics,
School of Fundamental Science and Engineering, \\
Waseda University, \\
3--4--1 Okubo, Shinjuku-ku, Tokyo, 169--8555, Japan \\
e-mail : \texttt{kumakawa@ruri.waseda.jp}

\end{flushleft}

\end{document}